\documentclass[10pt]{IEEEtran}
\usepackage{graphicx,epstopdf,multirow,multicol,subfigure,booktabs}

\usepackage{amsmath,bm}
\usepackage{amssymb,amsthm}
\usepackage{algorithm}
\usepackage[noend]{algpseudocode}
\usepackage{color}
\usepackage{pgfplotstable}
\usepackage{pgfplots}
\pgfplotsset{compat=newest}
\usepackage{float}
\usepackage{setspace}
\usepackage{cite}
\usepackage{times}
\usepackage{upgreek}
\usepackage{tikz}
\usetikzlibrary{shapes,shapes.misc}
\usepackage{tkz-graph}

\algnewcommand\algorithmicinput{\textbf{Input:}}
\algnewcommand\Input{\item[\algorithmicinput]}

\algnewcommand\algorithmicoutput{\textbf{Output:}}
\algnewcommand\Output{\item[\algorithmicoutput]}

\usepackage[colorlinks,citecolor=magenta,backref=page,linkcolor=cyan]{hyperref}
\usepackage{booktabs}
\allowdisplaybreaks
\graphicspath{figures/}

\theoremstyle{remark}
\newtheorem{proposition}{Proposition}

\newtheorem{definition}{Definition}

\makeatletter
\def\BState{\State\hskip-\ALG@thistlm}
\makeatother
\makeatletter
\def\munderbar#1{\underline{\sbox\tw@{$#1$}\dp\tw@\z@\box\tw@}}
\makeatother



\begin{document}
\title{\bf \LARGE Cooperative Routing for an Air-Ground Vehicle Team -- \\ Exact Algorithm, Transformation Method, and Heuristics}

\author{
  Satyanarayana G. Manyam$^{1}$, Kaarthik Sundar$^{2}$, and David W. Casbeer$^{3}$ %
\thanks{$^{1}$Research Scientist, Infoscitex corp., a DCS company, Dayton OH, 45431,  \texttt{msngupta@gmail.com}, }%
 \thanks{$^{2}$Center for Nonlinear Studies, Los Alamos, NM, 87544, \texttt{kaarthik01sundar@gmail.gov},}%
 \thanks{$^{3}$Technical Area Lead, Cooperative \& Intelligent Control, Control Science Center of Excellence, Air Force Research Laboratory, WPAFB, OH, 45433, \texttt{david.casbeer@us.af.mil}.}
}

\markboth{Journal of \LaTeX\ Class Files,~Vol.~14, No.~8, April~2018}%
{Gupta \MakeLowercase{\textit{et al.}}: Cooperative Vehicle Routing for an Aerial-Ground Vehicle Team}

\maketitle

\begin{abstract}
This article considers a cooperative vehicle routing problem for an intelligence, surveillance, and reconnaissance mission in the presence of communication constraints between the vehicles. The proposed framework uses a ground vehicle and an Unmanned Aerial Vehicle (UAV) that travel cooperatively and visit a set of targets while satisfying the communication constraints. The problem is formulated as a mixed-integer linear program, and a branch-and-cut algorithm is developed to solve the problem to optimality. Furthermore, a transformation method and a heuristic are also developed for the problem. The effectiveness of all the algorithms is corroborated through extensive computational experiments on several randomly generated instances. \\

\emph{Note to practitioners:} This paper is motivated by an intelligence, surveillance, and reconnaissance mission involving a single UAV and a ground vehicle, where the vehicles must coordinate their activity in the presence of communication constraints. The combination of a small UAV and a ground vehicle is an ideal platform for such missions, since small UAVs can fly at low altitudes and can avoid obstacles or threats that would be problematic for the ground vehicle alone. Furthermore, small UAVs can also be hand launched in difficult to reach areas and rough terrain. This paper addresses the coordinated routing problem involving these two vehicles and presents an algorithm to obtain an optimal solution for this problem, fast heuristics to obtain good feasible solutions, and also a transformation method to transform any instance of this cooperative vehicle routing problem to an instance of the one-in-a-set traveling salesman problem.

\end{abstract}


\IEEEpeerreviewmaketitle

\section{Introduction} \label{sec:intro}
\IEEEPARstart{O}{ver} the past two decades, UAVs are being used routinely in both military and civilian applications such as reconnaissance and surveillance expeditions, border patrol, weather and hurricane monitoring, crop monitoring etc. (see \cite{Frew2009, Curry2004, ZajkowskiT2006} and references therein). They are prime candidates for intelligence, surveillance, and reconnaissance (ISR) missions due to their several advantages such as portability and low risk, to name a few. A typical ISR mission would require the UAVs to collect images, videos, or sensor data and transmit them to a base station. The data collected in these ISR missions are most often very time sensitive and ideally, the data would be useful only if it is processed in real-time or near real-time. We propose a framework to meet this objective, and solve the underlying cooperative routing problem for the UAV and the ground vehicle. 


In this article, we present a cooperative vehicle routing problem for an ISR mission involving a UAV and a ground vehicle. The main motivation of using a platform consisting of a ground vehicle equipped with a UAV is the lack of direct roads and the presence of geographical obstacles viz. rivers, lakes, mountains, etc. that must be circumvented to reach the target locations. During such instances, the ground vehicle may not be able to reach certain targets and the UAV could be used to gather information from those targets. However, due to size, weight, and power restrictions the UAV would be unable to transmit the data from remote targets to the base station. Furthermore, in an ideal situation, the ground vehicle would have a communication link to the base station, perhaps through satellite or by other means. The proposed methodology enforces a communication constraint between the UAV and the ground vehicle that work cooperatively to accomplish the mission. Hence, in this paper, the UAV is required to stay connected to the ground vehicle, under the pretext that the data can be instantaneously transmitted to the base station through the ground vehicle. We also do not impose any specific constraint on communication for the ground vehicle, and thus in cases where communication is degraded and the ground vehicle cannot maintain contact with the base station, the cooperative routing problem will still yield a viable solution. In such a scenario, one could think of this problem as a cooperative routing problem to minimize the effort taken by the vehicles to acquire the data and then ferry it back to the base station. The problem is formally defined in the following section. 

\begin{figure}[h]
\centering
\includegraphics[height=2in]{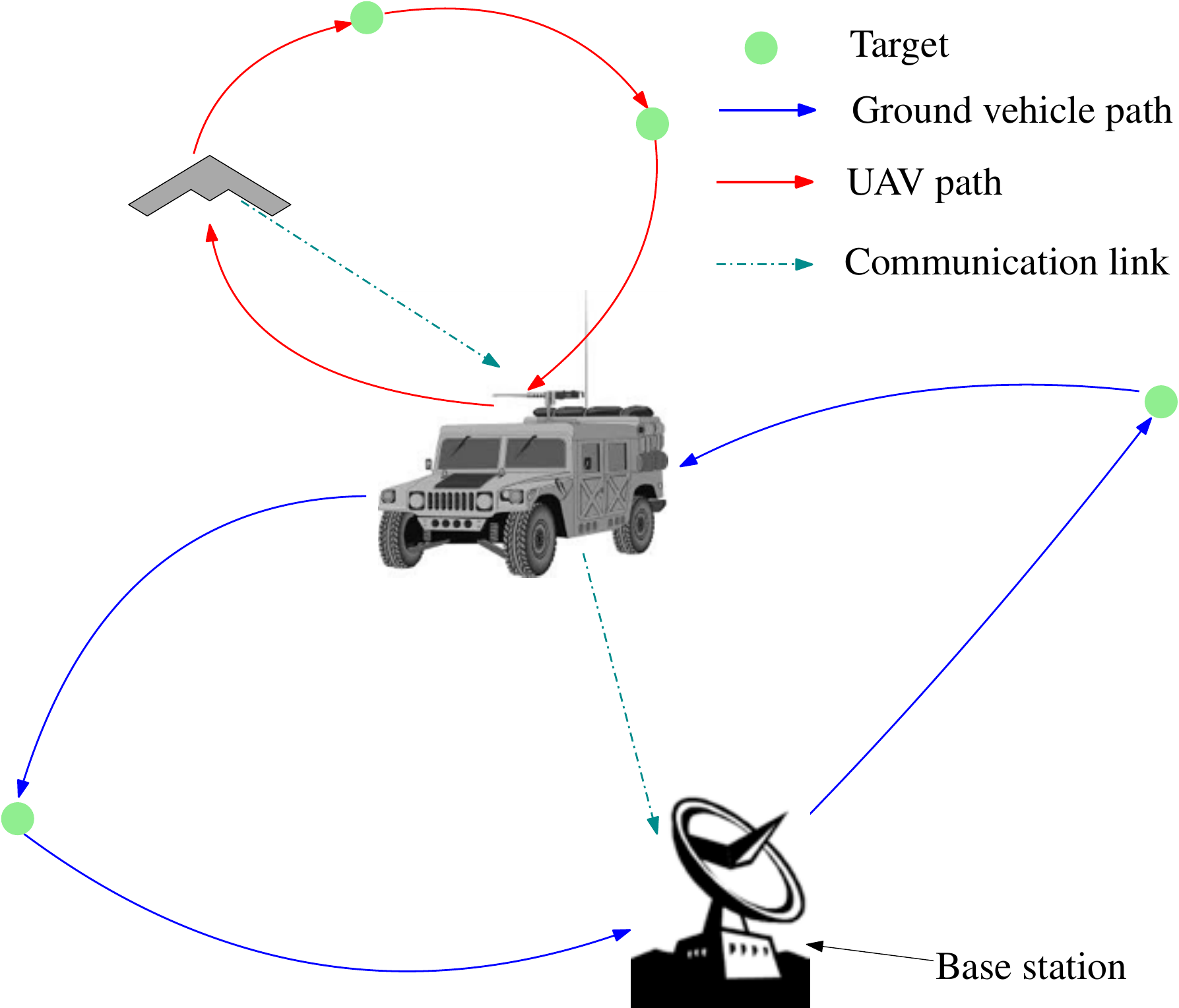}
\caption{The paths for the UAV and the ground vehicle.}
\label{fig:agvroute}
\end{figure}

\subsection{Cooperative routing problem} \label{subsec:intro_problem}
The problem considered in this article is formally defined as follows: we are given the locations of a set of targets and a base station where the ground vehicle is initially stationed. The UAV is carried by the ground vehicle. For each vehicle, we are given a travel cost between every pair of targets. This cost could be the euclidean distance between the pair of targets or could also depend on the terrain. Given this data, the objective of the problem is to determine paths of minimum cost for the ground vehicle and the UAV such that following conditions are satisfied: every target is visited either by the ground vehicle or by the UAV, and the UAV can always establish a reliable communication link with the ground vehicle.

The sequence of actions for a feasible mission is as follows: The ground vehicle starts at the base station and visits the first target in its route, then the UAV is deployed from the first target and it collects the data from a subset of targets according to its plan and returns to the ground vehicle. Then the ground vehicle proceeds to the next stop. This process is repeated until the data is collected from all of the targets. Once all the targets are visited, the ground vehicle carrying the UAV returns to the base station. An illustration of a typical route is shown in Fig. \ref{fig:agvroute}. As shown in the figure, any feasible path for the ground vehicle would start at the base station, make stops at a subset of targets, and return to the base station. As for the UAV, its feasible paths would be sub-tours rooted at the ground vehicle stops and satisfying the communication constraints. In particular, the objective of determining the routes for the vehicles involves (i) identifying the stops and order to visit them for the ground vehicle and (ii) at each stop, identifying the subset of targets and order in which the UAV has to visit them. Furthermore, to ensure that the UAV can always establish a communication link with the ground vehicle, we enforce the constraint that the UAV has to stay within a distance of $R$ units from the ground vehicle at all times during the mission. We shall, from here on, refer to this problem the \emph{cooperative aerial-ground vehicle routing problem} (CAGVRP). 

\subsection {Related work} \label{subsec:lit}
Cooperative control and path planning for a team of vehicles (UAVs or UAVs together with ground vehicles) has been a problem of interest and has received wide attention over the past decade (see \cite{rasmussenbook, kumarram, tsourdosbook, gilsparks, lasfarg, smithhdp, sujithicuas}). The CAGVRP is NP-hard because it is a generalization of the traveling salesman problem. Authors in \cite{smithhdp} propose a framework similar to the CAGVRP involving a truck traveling on a road network and a quadrotor for a package delivery problem in urban environments. These authors present a transformation algorithm to transform the problem to a generalized traveling salesman problem. This problem differs from the CAGVRP in two aspects: (i) the ground vehicle or the truck is restricted to travel along a road network and (ii) the quadrotor has to return to the ground vehicle after each delivery. Furthermore, the present paper differs from \cite{smithhdp} in that in addition to developing an analogous transformation algorithm, we present tailored algorithms to obtain an optimal solution to the CAGVRP and fast heuristics to obtain good feasible solutions to the problem. CAGVRP also resembles the two echelon vehicle routing problem \cite{perbolitwoech}; the key difference is in CAGVRP the locations at which the UAV tours originate are not known, which makes it even harder to solve.


Other problems addressed in the literature that are similar to the CAGVRP are the ring-star problem \cite{laportersp, kaarthikmrsp, baldaccimrsp, baldaccicaprsp} and the hierarchical ring-network problem \cite{golddam, altinkcor, carrollinoc, klincreview, shiself, stidsenringnet}. The ring-star problem aims to find a minimum cost cycle (or ring) through a subset of targets, and the targets that do not lie on the cycle should be assigned to one of the targets in the cycle. Algorithms to obtain an optimal solution based on the branch-and-cut paradigm are presented for the ring-star problem in \cite{laportersp}, for its multiple-depot variant in \cite{kaarthikmrsp}, and for the capacitated variants in \cite{baldaccimrsp, baldaccicaprsp}. In CAGVRP, not only must the targets that do not lie on the main ring be assigned to the targets on the ring, but sub-tours must also be chosen for all the off-ring targets. 

Another closely related problem is the hierarchical ring-network problem (HRNP) which aims to find a hierarchical two-layer ring-network. The top layer consists of a federal ring (analogous to the ground vehicle tour) which establishes connection between a number of node-disjoint metro rings (analogous to the UAV sub-tours) in the bottom layer. Authors in \cite{altinkcor} present heuristics and an approximation algorithm to solve the HRNP; they assume that the number of metro-rings and the nodes at which the metro rings are attached to the federal ring are given. In \cite{shiself}, heuristics and enumeration methods are given to solve the HRNP, where a certain demand must be satisfied between every pair of targets. Heuristics to solve a variant of the HRNP are presented in \cite{golddam}. 

A variant of HRNP, where the bottom layer could be rings or ring-stars is solved in \cite{carrollinoc}. The authors assume the number of possible local rings are given. The authors in \cite{stidsenringnet} use branch-and-price algorithm to solve the HRN problem with demands to be satisfied between every pair of targets. {Their solution procedure involved two steps: in the first step, they solve a modified HRNP which ignores the cost of the federal ring and design the metro rings; in the second step, they solve a generalized traveling salesman problem on the metro rings.} In all the aforementioned papers, the cost of the federal ring is either ignored or assumed to be equal to that of the metro ring. The CAGVRP differs from the HRNP in the following aspects: (i) the travel costs for the ground vehicle and the UAV are different, (ii) we have an additional distance constraint \emph{i.e.,} the distance between the UAV and the ground vehicle has to be within $R$ units throughout the mission, and (iii) we attempt to solve the coupled problem of finding paths for the ground vehicle and the UAV to minimize the total travel cost. Lastly, we point out that a preliminary version of this article that introduces the problem and develops a MILP formulation is published as a conference article \cite{Manyam2016}.

\subsection{Contributions:} \label{subsec:contributions}
The following are the main contributions of this article: (i) a mixed-integer linear programming (MILP) formulation and additional valid inequalities to strengthen the linear programming (LP) relaxation of the CAGVRP are developed, (ii) a branch-and-cut algorithm to compute an optimal solution to any instance of the problem based on the formulation is presented, (iii) an LP rounding-based heuristic is developed to increase the speed of the convergence of the branch-and-cut algorithm, (iv) a transformation algorithm to transform any instance of the CAGVRP to a corresponding instance of the one-in-a-set traveling salesman problem is developed, and finally (v) extensive computational results to corroborate the effectiveness of all the algorithms are presented. The rest of the article is organized as follows: we mathematically formulate the problem as a MILP in Sec. \ref{sec:prbfrm}. In Sec. \ref{sec:bandc} and Sec. \ref{sec:trans_alg}, we develop the branch-and-cut algorithm and the transformation algorithm for the CAGVRP, respectively. And finally, we present extensive computational results in Sec. \ref{sec:results} followed by conclusions and scope for future work in Sec. \ref{sec:conc}.

\section{Problem formulation} \label{sec:prbfrm}
We first introduce the required notations that are later used to formulate the CAGVRP as an MILP. 
\subsection{Notations} \label{subsec:notations}
Let $T$ denote the set of targets $\{0,\dots,n\}$; $0$ is the base station. The CAGVRP is defined on a mixed graph $G=(T, E\bigcup A)$, where $E$ and $A$ are a set of undirected edges and directed arcs, respectively, joining any two targets in $T$. Each edge $e = (i,j) \in E$ is associated with a non-negative cost $c_e$ required for the ground vehicle to traverse the edge $e$ (if $e$ connects the vertices $i$ and $j$, then $(i,j)$ and $e$ will be used interchangeably to denote the same edge in $E$). Similarly, each arc $[i,j] \in A$ is associated with a non-negative cost $d_{ij}$ required for the UAV to travel from target $i$ to target $j$; we remark that $A$ is allowed to have self loops \emph{i.e.}, $[i,i] \in A$ for every $i\in T$. We assume that the travel cost functions for the ground vehicle and the UAV satisfy the triangle inequality, \textit{i.e.}, for any distinct targets $i$, $j$, and $k$, $c_{ij} \leqslant c_{ik} + c_{k j}$ and $d_{ij} \leqslant d_{ik} + d_{k j}$, respec. We also associate with each arc $[i,j]\in A$, an auxiliary cost $f_{ij}$ to enforce the communication constraint that the UAV should always be within a distance of $R$ units from the ground vehicle. For any $i,j \in T$, $f_{ij}$ is set to $0$ if the Euclidean distance between targets $i$ and $j$ is less than $R$ units, and is set to an arbitrarily large value otherwise. We remark that this communication constraint can be enforced without the use of this auxiliary cost $f_{ij}$ for every arc $[i,j] \in A$ by introducing additional constraints. We choose the former to keep the presentation of the formulation cleaner. Additionally, for any subset of targets $S \subseteq T$, we let $\delta(S) \subseteq E$ denote the set of edges with one end in the set $S$ and another end in $T\setminus S$; similarly, we let $\delta^+(S) \subseteq A$ and $\delta^-(S) \subseteq A$ denote the set of arcs directed into and away from the set $S$, respectively. The sets $\delta(S)$, $\delta^+(S)$, and $\delta^-(S)$ are also referred to as cut-sets in the literature \cite{tothvigo,sundar2016generalized}. Finally, given $S\subseteq T$, we let $\gamma(S)$ denote the set of all the edges with both end points in $S$. 

We let $\mathcal F$ denote any feasible solution to the CAGVRP and associate vector of decision variables $\bm x \in \{0,1\}^{|E|}$, $\bm w \in \{0,1\}^{|A|}$, and $\bm y \in \{0,1\}^{|A|}$ with each $\mathcal F$. The component $x_e$ of $\bm x$, associated with edge $e \in E$, is a binary variable which takes a value $1$ if the edge $e$ is traversed by the ground vehicle, and $0$ otherwise. Each component $w_{ij}$ of $\bm w$, associated with the directed arc $[i,j] \in A$, is a binary variable that indicates the presence of the arc in a UAV sub-tour. Similarly, $y_{ij}$, a component of $\bm y$ is a binary assignment variable that takes a value $1$ if the target $i$ is assigned to the UAV sub-tour that whose root is the target $j$, and $0$ otherwise. Note that if a target $i$ is visited by a ground vehicle then the assignment variable $y_{ii}$ is equal to $1$ (the target is assigned to itself). 

\subsection{Mixed-integer programming formulations for CAGVRP} \label{subsec:mip}
With the notations and decision variables introduced thus far, CAGVRP is first formulated as mixed-integer program (MIP), whose objective minimizes the total cost of travel for the ground vehicle and the UAV. The objective is given by the following equation: 
\begin{flalign}
& \min ~\sum_{e\in E} c_e x_e + \sum_{[i,j] \in A} \left( d_{ij} w_{ij} + f_{ij} y_{ij} \right) \label{eq:obj}
\end{flalign}
The auxiliary cost term $f_{ij}y_{ij}$ for each arc $[i,j] \in A$ in the objective function implicitly enforces the constraint that the UAV should always be within a distance of $R$ units from the ground vehicle. We remark that $f_{ij}$ can also by used to denote the cost of communication; but, we do not do so throughout the rest of the paper, for ease of exposition. The MILP is subjected to the constraints detailed below.

\noindent \textit{\underline{\smash{Assignment constraints}}}: Here the constraints on communication between the UAV and the ground vehicle, while it is stopped at a target, are formulated. If a target $i$ is assigned to itself, then it corresponds to the case when the target being visited by the ground vehicle and if the target $i$ is assigned to another target $j$, then it corresponds to the case when the target $i$ is visited by the UAV when the ground vehicle has a stop at target $j$ and the UAV communicates the data it collects from target $i$ to the ground vehicle at target $j$. The assignment constraints are given as follows:
\begin{flalign}
& \sum_{j \in T} y_{ij} = 1 \quad \forall i \in T, \text{ and} & \label{eq:assignment1} \\
& y_{0 0} = 1. & \label{eq:assignment2}
\end{flalign}
The constraint \eqref{eq:assignment2} enforces the base station $0$ is visited by the ground vehicle; this constraints conforms with our assumption that the ground vehicle is initially stationed at the base station. 

\noindent \textit{\underline{\smash{Degree constraints}}}: 
The degree constraints in \eqref{eq:degx} ensure that the number of undirected ground vehicle edges incident on any target $i\in T$ is equal to $2$ if and only if the target $i$ is assigned to itself ($y_{ii} = 1$) \textit{i.e.}, the target is visited by the ground vehicle. The in-degree and out-degree constraints for the UAV routes (sub-tours rooted at targets where the ground vehicle stops) are enforced using the Eq. \eqref{eq:indegw} and \eqref{eq:outdegw}, respectively. These constraints ensure that, for \textit{every} target there are two UAV edges in the UAV route, one entering and another leaving the target. This is in contrast to the ground vehicle routes, where edges are incident on a target only when the target is assigned to itself. The UAV has to visit every target $i\in T$ irrespective of the value of $y_{ii}$. When the value of the $w_{ii}=1$, the constraints in \eqref{eq:outdegw} and \eqref{eq:indegw} would result in a trivial UAV sub-tour, which occurs when the UAV is not deployed when the ground vehicle stops at target $i$. The degree constraints are as follows:
\begin{flalign}
& \sum_{j \in T} x_{ij} = 2 y_{ii} \quad \forall i \in T,  &\label{eq:degx} \\
& \sum_{j \in T} w_{ij} = 1, \quad \forall i \in T, \text{ and } &\label{eq:outdegw} \\
&\sum_{i \in T} w_{ij} = 1, \quad \forall j \in T. &\label{eq:indegw}
\end{flalign}

\noindent \textit{\underline{\smash{Connectivity constraints}}}: The degree constraints and the assignment constraints alone are not sufficient to guarantee a feasible (connected) solution to the CAGVRP. An instance of such an infeasible solution is shown in Fig. \ref{fig:infeas}. In general, there are three predominant techniques to ensure connected solutions. The first is to introduce an additional set of constraints viz. the Miller-Tucker-Zemlin (MTZ) constraints, the second way is to use a multi- or single- commodity flow-based constraints, and the final way is to enforce connectivity constraints using cut-sets \cite{applegatebook}. For vehicle routing problems, the cut-set based connectivity constraints are known to outperform MTZ and flow constraints computationally\cite{tothvigo,sundar2015exact,sundar2016generalized,sundar2017algorithms}.
\begin{flalign}
&\sum_{e \in \delta (S)}x_e \geqslant 2 \sum_{j \in S}y_{ij}, \quad \forall i \in S, \, 0 \notin S, \, S \subset T, & \label{eq:xsec}\\
&\sum_{[i,j] \in \delta^+(S)} w_{ij} \geqslant 1-\sum_{j \in S} y_{ij}, \quad \forall i \in S, \, S \subset T, & \label{eq:wsec1}\\
&\sum_{[i,j] \in \delta^-(S)} w_{ij} \geqslant 1-\sum_{j \in S} y_{ij}, \quad \forall i \in S, \, S \subset T, &\label{eq:wsec2} 
\end{flalign}
Therefore, we utilize the constraints in \eqref{eq:xsec}, \eqref{eq:wsec1} and \eqref{eq:wsec2}, that are formulated using cut-sets, to ensure connectedness of the solutions. The constraints in \eqref{eq:xsec} ensure that each UAV sub-tour is rooted to at least one target in the ground vehicle path and the constraints \eqref{eq:wsec1} and \eqref{eq:wsec2} mandate the connectedness of the ground vehicle path. We remark that the number of connectivity constraints in Eqs. \eqref{eq:xsec} -- \eqref{eq:wsec2} are exponential and a dynamic cut-generation algorithm, that adds constraints only when they are violated, is detailed in the forthcoming section to handle these constraints in a computationally efficient manner. 
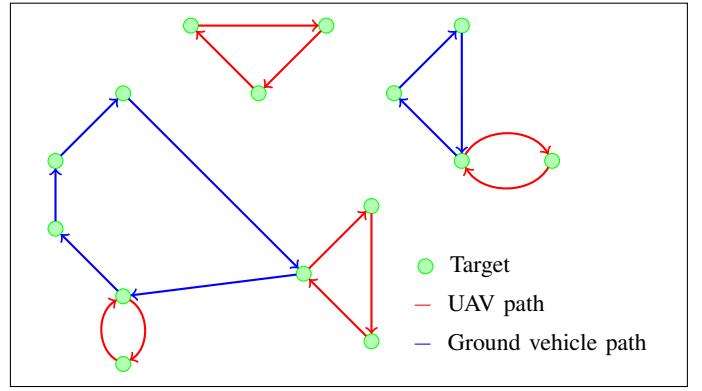
\begin{figure}
\centering
\begin{tikzpicture}[scale=0.6]
\draw [thin] (-2.5,-0.5) rectangle (12.5,8);
\tikzset{VertexStyle/.style = {shape=circle,draw, inner sep=0.7mm,color=green,fill=green!30}} 
\SetVertexNoLabel
\SetGraphUnit{1.5}
\Vertex{1}
\NO(1){2}
\NOWE(2){3}
\NO(3){4}
\NOEA(4){5}
\SOEA[unit=4](5){6}
\NOEA(6){7}
\SOEA(6){15}
\EA[unit=3](5){8}
\NOWE(8){9}
\NOEA(8){10}
\EA[unit=3](8){11}
\NOEA(11){12}
\SO[unit=3](12){13}
\EA[unit=2](13){14}

\tikzstyle{target}=[circle,draw,color=green,fill=green!30,inner sep=0.7mm]
\tikzstyle{route}=[rectangle, minimum height=0pt, minimum width=6pt, inner sep=0pt, draw]
\tikzset{EdgeStyle/.style = {->,blue,thick}}
\Edges(2,3,4,5,6,2)
\Edges(11,12,13,11)
\tikzset{EdgeStyle/.style = {->,bend left=60,red,thick}}
\Edge(2)(1)
\Edge(1)(2)
\Edge(13)(14)
\Edge(14)(13)
\tikzset{EdgeStyle/.style = {->,red,thick}}
\Edges(6,7,15,6)
\Edges(8,9,10,8)

\matrix [draw=none,below left] at (12,2.8) {
  \node [target,label={[label distance=0.1cm]0:{\small Target}}] {}; \\
  \node [route,draw=red,label={[label distance=0.1cm]0:{\small UAV path}}] {}; \\
  \node [route,draw=blue,label={[label distance=0.1cm]0:{\small Ground vehicle path}}] {}; \\
};
\end{tikzpicture}
\caption{An infeasible solution to the CAGVRP caused due to disconnected paths for both the UAV and the ground vehicle.}
\label{fig:infeas}
\end{figure}

\noindent \textit{\underline{\smash{UAV sub-tour constraints}}}: The UAV sub-tour constraints ensure that a UAV can traverse an arc between any two targets $i$ and $j$ only if both the targets are assigned to the same (ground-vehicle) target $k \in T$. Mathematically, this condition is enforced using the following nonlinear constraint:
\begin{flalign}
&w_{ij}  \leqslant \sum_{k \in T} y_{ik} y_{jk}, \qquad \forall [i,j] \in A. &\label{eq:wij}
\end{flalign}

In summary, the MIP formulation for the CAGVRP is given by the formulation $\mathcal F_1$, as follows:
\begin{flalign*}
& (\mathcal F_1):~ \text{Minimize Eq. \eqref{eq:obj}, subject to: Eqs. \eqref{eq:assignment1} -- \eqref{eq:wij}, } & \\
& \text{$\bm x \in \{0,1\}^{|E|}$, $\bm w \in \{0,1\}^{|A|}$, and $\bm y \in \{0,1\}^{|A|}$. } &
\end{flalign*}
As mentioned previously, the formulation $\mathcal F_1$, presented above, is non-linear due to the constraints in \eqref{eq:wij}. These constraints can be linearized by introducing an auxiliary variable $z_{ijk}$ for every triplet of targets $i,j,k\in T$. 
\begin{proposition} \label{prop:linearization}
The constraints in \eqref{eq:wij} permits the following linear reformulation: 
\begin{flalign}
&w_{ij} \leqslant \sum_{k \in T} z_{ijk} \qquad \forall [i,j] \in A, &\label{eq:wijlin}\\
&z_{ijk} \leqslant y_{ik} \qquad \forall i,j,k \in T, &\label{eq:zlin1}\\
&z_{ijk} \leqslant y_{jk} \qquad \forall i,j,k \in T, \text{ and} &\label{eq:zlin2} \\
&z_{ijk} \geqslant y_{ik} + y_{jk} -1 \qquad \forall i,j,k \in T.& \label{eq:zlin3} 
\end{flalign}
\end{proposition}
\begin{proof}
We prove the proposition by considering the following two cases (i) $y_{ik} = y_{jk} = 1$ and (ii) $y_{ik} = 0$ or $y_{jk} = 0$. For case (i), observe that $z_{ijk}$ will take a value of $1$ due to the constraints \eqref{eq:zlin1}--\eqref{eq:zlin3}. Similarly for case (ii), suppose $y_{ik} = 1$ and $y_{jk} = 0$, then $z_{ijk} = 0$. A similar argument holds when $y_{jk} = 1$ and $y_{ik} = 0$. 
\end{proof}

Using the above proposition \ref{prop:linearization}, the MIP formulation $\mathcal F_1$ can be equivalently reformulated into a MILP given by:
\begin{flalign*}
&(\mathcal{F}_2):~ \text{Minimize Eq. \eqref{eq:obj}, subject to: Eqs. \eqref{eq:assignment1} -- \eqref{eq:wsec2},} & \\
& \text{Eqs. \eqref{eq:wijlin} -- \eqref{eq:zlin3}, $z_{ijk} \in \{0,1\} \quad \forall i,j,k \in T$ and } & \\
& \text{$\bm x \in \{0,1\}^{|E|}$, $\bm w \in \{0,1\}^{|A|}$, and $\bm y \in \{0,1\}^{|A|}$. } &
\end{flalign*}
If the binary restrictions on all the variables of a MILP are relaxed, then we call that model a LP relaxation of the MILP. Let $\mathcal F_2^{lp}$ denote the LP relaxation of $\mathcal F_2$. In the following subsection, we shall strengthen $\mathcal F_2^{lp}$ by introducing additional valid inequalities; a constraint is called a valid inequality if it does not remove any feasible solution to the MILP.

\subsection{Additional valid inequalities} \label{subsec:vi}
We first adapt the well-known $2$-matching inequalities that is known to be valid for the Traveling Salesman Problem (TSP) \cite{grotpadpoly,grotschelhollandmp,fischettigtsp} and a variety of other vehicle routing problems related to the CAGVRP \cite{laportersp,kaarthikmrsp,sundar2015exact}. Specifically, we consider the following inequality: 
\begin{flalign}
\sum_{e \in \gamma(H)}x_e + \sum_{e \in I}x_e \leqslant &\sum_{i \in H} y_{ii} + \frac{|I|-1}{2}, \label{eq:twomatch} 
\end{flalign}
for all $H \subseteq T$ and $I \subseteq \delta(H)$. Here, $H$ is called the handle and $I$, the teeth. $H$ and $I$ satisfy the following conditions: (i) no two edges in the teeth are incident on the same target and (ii) the number of edges in the teeth is odd and greater than equal to $3$. An interested reader is referred to \cite{laportersp,kaarthikmrsp,sundar2015exact} for a proof of validity of the above inequality.

The next set of inequalities enforce additional restrictions on the UAV and ground vehicle paths. Before presenting these inequalities, we introduce an additional notation. For each target $i \in T$, we let $R_{i} \subseteq T$ denote the subset of targets that are within a distance of $R$ units (communication range) from $i$. Given this notation, the following set of inequalities are valid for the CAGVRP:
\begin{flalign}
& \sum_{i \in R_{j}} y_{i j} \leqslant |R_{j}|y_{j j} \quad \forall j \in T, & \label{eq:assign_vi} \\
& w_{i j} \leqslant 2 - y_{i i} - y_{j j} \quad \forall [i, j] \in A. & \label{eq:gv_vi}
\end{flalign}
Eq. \eqref{eq:assign_vi} bounds the number of targets the UAV can visit in a sub-tour based on the root node of the sub-tour and the Eq. \eqref{eq:gv_vi} ensures that the UAV does not traverse an edge that is already been traversed by the ground vehicle. In the next section, we present an overview of the branch-and-cut algorithm to solve the MILP formulation $\mathcal F_2$ to optimality. 

\section{Branch-and-cut algorithm} \label{sec:bandc}
In this section, we briefly present the main ingredients of a branch-and-cut algorithm that is used to solve the formulation $\mathcal F_2$ to optimality. The formulation $\mathcal F_2$ for the CAGVRP can be provided to off-the-shelf commercial branch-and-bound solvers like CPLEX \cite{cplexmanual} to obtain an optimal solution. These solvers typically require the complete formulation to be provided to the solver a priori, which in our case would not be computationally tractable due to the exponential number of connectivity constraints in Eqs. \eqref{eq:xsec} -- \eqref{eq:wsec2} and the 2-matching constraints in Eq. \eqref{eq:twomatch}. To address this issue of exponential number of constraints, we utilize the following approach: we relax these constraints from
the formulation. We let $\operatorname{relax}(\mathcal F_2)$ denote this relaxed formulation and $\operatorname{relax}(\mathcal F_2^{lp})$, its LP relaxation. Whenever the solver obtains a feasible solution to $\operatorname{relax}(\mathcal F_2)$ or fractional feasible solution $\operatorname{relax}(\mathcal F_2^{lp})$, we check if any of
these connectivity constraints or $2$-matching constraints are violated by the feasible solution, integer or fractional. If so, we add the constraint dynamically and
continue solving the original problem. This process of adding constraints to the problem sequentially has been observed to be computationally efficient for the traveling salesman problem and the variants of CAGVRP \cite{tothvigo,sundar2015exact}. The algorithms used to identify violated constraints given a fractional or integer feasible solutions to $\operatorname{relax}(\mathcal F_2)$ or $\operatorname{relax}(\mathcal F_2^{lp})$, respectively, are referred to as separation algorithms. The branch-and-cut solvers like CPLEX provide ways to interact with the branch-and-cut solution process and modify its behaviour through external user-defined functions referred to as callback functions. We dynamically identify and add violated constraints given a fractional or an integer solution to the relaxed problem using the callback functions known as the ``user-cut'' or the ``lazy constraint'' callbacks, respectively. A flow-chart of the modifications done to the actual branch-and-cut solution process is shown in Fig. \ref{fig:flowchart-bc}. 

\begin{figure}
    \centering
    \includegraphics[scale=0.5]{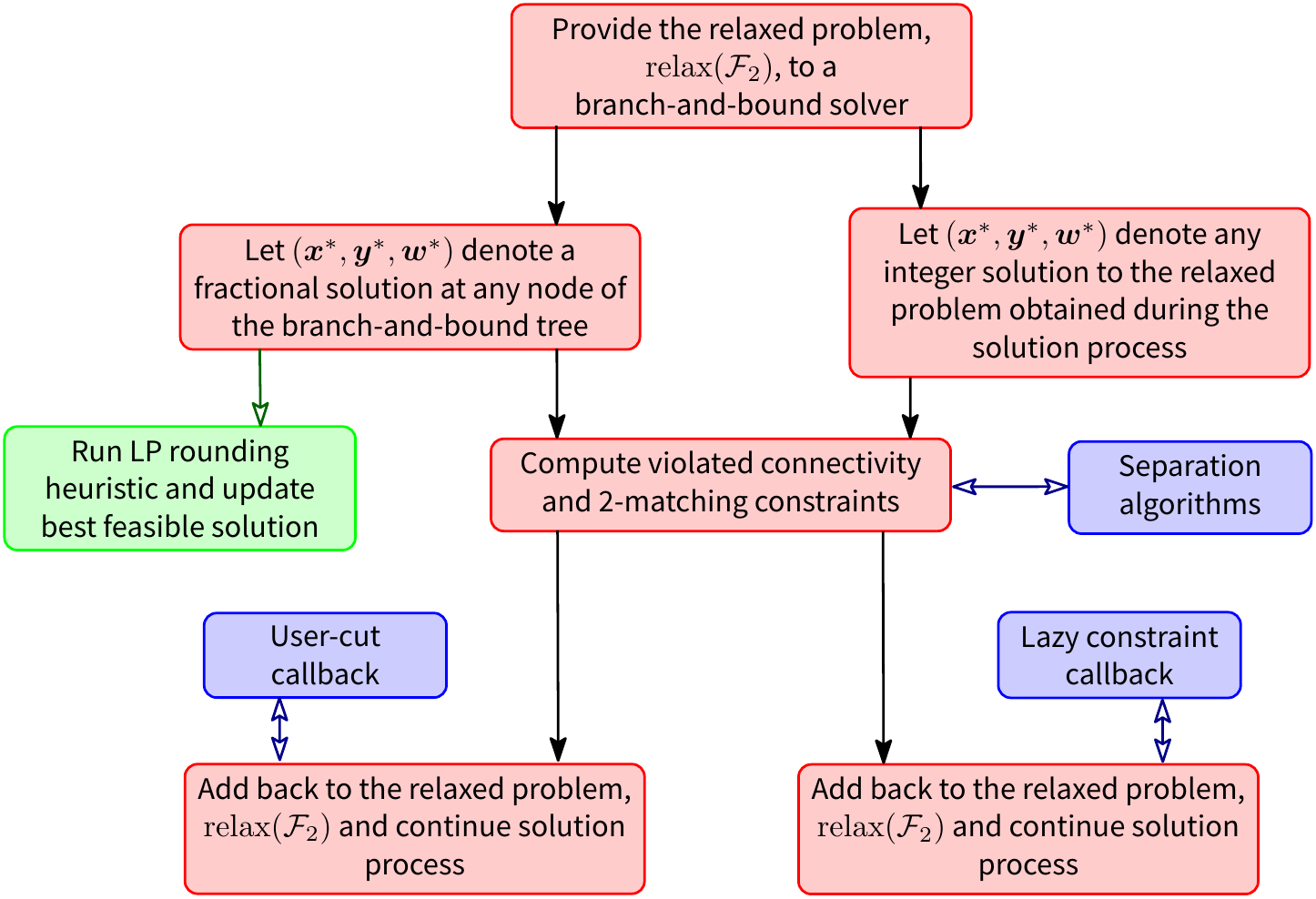}
    \caption{Flow chart of the modifications to change the behaviour of the traditional branch-and-cut algorithm. The LP rounding heuristic, shown in the green box, is used to convert fractional solutions to a feasible CAGVRP solution and is detailed in Sec. \ref{subsec:lprounding}.}
    \label{fig:flowchart-bc}
\end{figure}

The separation algorithms in the forthcoming sections are the algorithms that identify violated connectivity and $2$-matching constraints given integer or fractional solutions to the corresponding relaxed problems. 

\subsection{Separation algorithms} \label{subsec:separation}
This section presents the separation algorithms for identifying violated connectivity constraints and $2$-matching constraints given a solution to the relaxed problems. A separation algorithm is termed as ``exact'' if it is guaranteed to find a violated constraint, if there exists one. Otherwise, the separation algorithm is referred to as a heuristic separation algorithm. \\

\noindent \textit{\underline{\smash{Separation of connectivity constraints in \eqref{eq:xsec} -- \eqref{eq:wsec2}}}}: We first discuss exact separation algorithms for identifying violated connectivity constraints in Eqs. \eqref{eq:xsec} -- \eqref{eq:wsec2} given a fractional solution; the algorithm trivially extends to the case with integer solution. To this end, let $(\bm x^*, \bm y^*, \bm w^*)$ denote a fractional solution. We first construct a support graph ($G^*$) based on the fractional solution; $G^*=(V^*, E^*)$, $V^* = \{i \in T: 0 < y^*_{ii} < 1 \}$ and $E^* = \{e \in E: 0 < x^*_e < 1 \}$. Then, we check if the graph $G^*$ is connected; if it is not connected, each vertex set $S$ corresponding to an individual connected component, such that $0 \notin S$, generates a violated constraint (\ref{eq:xsec}) for each $i \in S$. If $G^*$ is connected, we find the subset of nodes $S$ with minimum value of $\sum_{e \in \delta(S)} x^*_e$. This is done by solving a problem of computing the max-flow on $G^*$, where the capacity of each edge $e$ is set to $x^*_e$. We compute the minimum cut of all pairs of nodes on $G^*$ and consider the node partition $S$ that does not contain $0$. We check if this set violates  constraint \eqref{eq:xsec} for each $i \in S$. If the constraint is violated for any $i$, we add the corresponding inequality to the existing pool of inequalities. 

We use a similar procedure to find violated connectivity constraints \eqref{eq:wsec1} -- \eqref{eq:wsec2} for the UAV sub-tours. Based on the fractional solution $\bm w^*$, we construct graph $G^*_u = (V^*, A^*)$, $A^* =\{ [i, j] \in A: 0<w_{ij}<1 \} $. Similar to the previous procedure we find the violated constraints \eqref{eq:wsec1} -- \eqref{eq:wsec2} defined by $S$, such that the subset $S$ does not contain any targets that are visited by the ground vehicle. \\

\noindent \textit{\underline{\smash{Separation of $2$-matching constraints in \eqref{eq:twomatch}}}}: To find the violated $2$-matching constraints, we follow the heuristic separation algorithm described in \cite{laportersp, fischettigtsp}. We consider each connected component $H$ of $G^*$ as a handle of a possible violated 2-matching inequality, whose two-target teeth correspond to edges in $e\in \delta(H)$ with $x_e^* = 1$. We reject the inequality if the number of teeth is even. If the inequality is violated, then we add this inequality to the pool of violated inequalities. We remark that though exact separation algorithms for the $2$-matching constraints exist in the literature \cite{padberg1982odd}, we utilize heuristic algorithms in favour of their computational effectiveness and algorithmic complexity. 

Next, we present a LP rounding heuristic that is applied to every fractional feasible solution to the CAGVRP to compute good feasible solutions during the branch-and-cut algorithm. It is well known in the literature \cite{laportersp, sundar2015exact,tothvigo} that computationally efficient heuristics that can be applied on the fractional solutions obtained during the execution of the branch-and-cut algorithm can potentially speed up the convergence of the algorithm itself. 

\subsection{LP rounding heuristic} \label{subsec:lprounding}
The rounding heuristic, presented in this section, is applied to every fractional solution obtained at each node of the branch-and-cut tree. To that end, let $(\bm x^*, \bm y^*, \bm w^*)$ denote any fractional solution obtained during the execution of the branch-and-cut algorithm. The algorithm then rounds the fractional assignment variable vector $\bm y^*$ to obtain a feasible set of assignments and utilizes the LKH \cite{lkhheur} heuristic for the TSP to compute paths for the ground vehicle and the UAV. Given $y_{ii}^*$ for each $i \in T$, the heuristic lets $i$ be a ground vehicle stop if $y_{i i}^* \geqslant 0.5$. All the other targets are assumed to be visited by the UAV, and feasible ground vehicle and UAV routes that satisfy the communication restrictions are then constructed. An overview of the heuristic is given by the algorithm \ref{alg:roundheur}. The pseudo-code outputs a heuristic solution vector $(\hat{\bm x}, \hat{\bm y}, \hat{\bm w})$, that corresponds to a feasible CAGVRP solution. 

\begin{algorithm}
    \setstretch{1.1}
  \caption{Pseudo-code of the LP rounding heuristic}
  \label{alg:roundheur}
  \begin{algorithmic}[1]
    \vspace{1ex}
    \Input fractional assignment variable values, $\bm y^*$
    \Output heuristic solution vector, $(\hat{\bm x}, \hat{\bm y}, \hat{\bm w})$
    \State \label{step:1} $S_g = \{i \in T: y_{i i}^* \geqslant 0.5\}$ and $S_a = T \setminus S_g$
    \State $\hat{y}_{ii} \gets 1$, $\forall i \in S_g$ and $\hat{y}_{ii} \gets 0$, $\forall i \in S_a$
    \For{$l \in S_g$} \label{step:3}
    $T^l \gets \emptyset$
    \EndFor
    \For{$u \in S_a$}
    \If{$R_u \cap S_g \neq \emptyset $} \label{step:5}
	\State \label{step:6} $k =  \operatorname{arg\,min}_{v \in R_u \cap S_g} d(u,v)$ \Comment{$d(\cdot,\cdot)$: distance}
	\State \label{step:7} $\hat{y}_{uk} \gets 1$, $T^k \gets T^k \cup \{ u\}$
	\Else
	\State  \label{step:9} $S_g \gets S_g \cup \{ u\}$, $S_a \gets S_a \setminus \{ u\}$
	\EndIf	
    \EndFor
    \State  \label{step:10} $\mathcal T_g \gets \operatorname{LKH}(S_g)$ \Comment{heuristic to solve TSP}
    \For{$l \in S_g$}  \label{step:11}
    $\mathcal T_a^l \gets \operatorname{LKH}(T^l)$
    \EndFor
    \State \label{step:12} $\hat x_e \gets 1$, $\forall e \in \mathcal T_g$
    \For{$l \in S_g$} \label{step:13} $\hat w_{ij} \gets 1$, $\forall [i,j] \in \mathcal T_a^l$
    \EndFor
    \State \label{step:14} $\hat{y}_{ii} \gets 1$, $\forall i \in S_g$ and $\hat{y}_{ii} \gets 0$, $\forall i \in S_a$
  \end{algorithmic}
\end{algorithm}
In Step \ref{step:1} of the pseudo-code in algorithm \ref{alg:roundheur}, $S_g$ and $S_a$ are subsets of the target set $T$ that are to be visited by the ground vehicle and the UAV, respectively; these subsets are constructed by rounding the the fractional assignment variable vector $\bm y^*$. In Step \ref{step:3}, the set $T^l$ represents the subset of targets that are in the UAV sub-tour rooted at the target $l$. In Step \ref{step:5}, $R_u$ is the set of targets that are with in the communication range of target $u$. The Steps \ref{step:6} and \ref{step:7} construct the sets by assigning each target in the set $S_a$ to one of the sets $T^k$, $k\in S_g$. If such an assignment is not possible for a target in the set $S_a$, the target is moved to the set $S_g$ \textit{i.e.}, that target is visited by the ground vehicle (Step \ref{step:9}). The LKH \cite{lkhheur} heuristic (to solve a TSP) is applied to the targets in the ground vehicle set $S_g$ and UAV sub-tour sets $T^l$, $l\in S_g$, in Steps \ref{step:10} and \ref{step:11}, respectively. Finally, the heuristic solution vector $(\hat{\bm x}, \hat{\bm y}, \hat{\bm w})$ are updated using the routes obtained via the LKH heuristic in the Steps \ref{step:12} -- \ref{step:14}. 

In the next section, we present a transformation algorithm to transform any instance of the CAGVRP to the one-in-a-set traveling salesman problem, also referred to in the literature as the generalized traveling salesman problem (GTSP). Once, the problem is transformed to the GTSP, efficient heuristics, available in the literature \cite{Smith2016}, can be used to quickly compute a feasible solution to the corresponding GTSP instance; this in turn can be converted to a feasible solution for the CAGVRP. The heuristic in \cite{Smith2016} a large neighborhood search algorithm for the GTSP which is known to compute close to optimal solutions in very little computation time.

\section{Transformation algorithm} \label{sec:trans_alg}
The NP-hardness of the CAGVRP restricts the use of the MILP model and the branch-and-cut algorithm to compute the optimal solution to small and medium-sized CAGVRP problem instances. Hence, the need to develop efficient heuristics for the CAGVRP. To that end, we present a transformation method that transforms the CAGVRP defined on the graph $G(T,E\bigcup A)$ into GTSP on a transformed graph $\overline G = (\overline V, \overline E)$. For the sake of completeness, we first informally define the GTSP as follows:
\begin{definition}{\it GTSP \cite{noonbean}:} The GTSP is defined on a directed graph in which the vertices have been pre-grouped into $m$ mutually exclusive and exhaustive vertex sets. Arcs are defined only between vertices belonging to different vertex sets with each arc having an associated cost. The GTSP is the problem of finding a minimum cost directed cycle which includes exactly one vertex from each vertex set.
\end{definition}
\subsection{Construction of the transformed graph $\overline G$} \label{subsec:construction}
We will next begin by introducing some definitions that will aid in defining the transformed graph, $\overline G$.
\begin{definition}{\it Configuration:} \label{defn:conf}
Given any pair of targets $i,j \in T$, we refer to $\mathcal C(g_i, a_j)$ as a configuration if the ground vehicle and the UAV are located at the targets $i$ and $j$, respectively. We refer to a configuration as a hub when $i=j$ (\emph{i.e.},  $\mathcal C(g_i, a_i)$). 
\end{definition}
\begin{definition}{\it Feasible configuration:} \label{defn:f_conf}
A configuration $\mathcal C(g_i, a_j)$, where $i,j \in  T$, is feasible if $d(i,j) \leqslant R$ \emph{i.e.}, the distance between the targets $i$ and $j$ is less than or equal to the communication range $R$. If a configuration does not satisfy this condition, then it is referred to as infeasible. 
\end{definition}
\begin{figure}
\centering
\subfigure[Edge between configurations $\mathcal C(g_i, a_i)$ and $\mathcal C(g_k, a_k)$ where both configurations are hubs. The cost of this edge is given by $c_{ik}$.]{\begin{tikzpicture}
\tikzstyle{target}=[circle,draw=green,fill=green!30,inner sep=0.4mm]
\tikzstyle{gv}=[star,star points=4,fill=blue!30,draw=blue,inner sep=0.4mm]
\tikzstyle{av}=[star,star points=6,fill=red!30,draw=red,inner sep=0.4mm]
\tikzstyle{route}=[rectangle, minimum height=0pt, minimum width=6pt, inner sep=0pt, draw]
\draw [thin] (-2,3) rectangle (6,5);
\draw [black,dashed] (2,5) -- (2,3);
\node [draw=none] at (0,3.3) {\small $\mathcal C(g_i,a_i)$};
\node [draw=none] at (4,3.3) {\small $\mathcal C(g_k,a_k)$};
\draw (-1,4.5) node (from_1) [target] {\small $i$};
\draw (1,4.5) node (to_1) [target] {\small $k$};
\draw (3,4.5) node (from_2) [target] {\small $i$};
\draw (5,4.5) node (to_2) [target] {\small $k$};
\draw [blue,->,thick] (from_2) -- (to_2);
\node [gv] at (-1.2,4.2) {};
\node [av] at (-0.8,4.2) {};
\node [gv] at (4.8,4.2) {};
\node [av] at (5.2,4.2) {};
\node [target,inner sep=0.7mm,label={[label distance=0.1cm]0:{\small Target}}] at (-1.8,5.3) {};
\node [gv,label={[label distance=0.1cm]0:{\small GV}}] at (-0.3,5.3) {};
\node [av,label={[label distance=0.1cm]0:{\small UAV}}] at (0.8,5.3) {};
\node [route,color=blue,label={[label distance=0.1cm]0:{\small GV path}}] at (2.1,5.3) {};
\node [route,color=red,label={[label distance=0.1cm]0:{\small UAV path}}] at (3.9,5.3) {};
\end{tikzpicture}\label{fig:gpcons1}}
\subfigure[Edge between configurations $\mathcal C(g_i, a_j)$ and $\mathcal C(g_i, a_{\ell})$. The first configuration can be a hub, while the second configuration cannot. The cost of the edge is given by $d_{j\ell}$.]{\begin{tikzpicture}
\tikzstyle{target}=[circle,draw=green,fill=green!30,inner sep=0.4mm]
\tikzstyle{gv}=[star,star points=4,fill=blue!30,draw=blue,inner sep=0.4mm]
\tikzstyle{av}=[star,star points=6,fill=red!30,draw=red,inner sep=0.4mm]
\tikzstyle{route}=[rectangle, minimum height=0pt, minimum width=6pt, inner sep=0pt, draw]
\draw [thin] (-2,3) rectangle (6,5);
\draw [black,dashed] (2,5) -- (2,3);
\node [draw=none] at (-1,3.3) {\small $\mathcal C(g_i,a_j)$};
\node [draw=none] at (3,3.3) {\small $\mathcal C(g_i,a_{\ell})$};
\draw (-1,4.5) node (gv_1) [target] {\small $i$};
\draw (0.5,4.5) node (from_1) [target] {\small $j$};
\draw (1,3.5) node (to_1) [target] {\small $\ell$};
\draw (3,4.5) node (gv_2) [target] {\small $i$};
\draw (4.5,4.5) node (from_2) [target] {\small $j$};
\draw (5,3.5) node (to_2) [target] {\small $\ell$};
\draw [red,->,thick] (from_2) -- (to_2);
\node [gv] at (-1.4,4.5) {};
\node [gv] at (2.6,4.5) {};
\node [av] at (0.1,4.5) {};
\node [av] at (4.6,3.5) {};
\node [target,inner sep=0.7mm,label={[label distance=0.1cm]0:{\small Target}}] at (-1.8,5.3) {};
\node [gv,label={[label distance=0.1cm]0:{\small GV}}] at (-0.3,5.3) {};
\node [av,label={[label distance=0.1cm]0:{\small UAV}}] at (0.8,5.3) {};
\node [route,color=blue,label={[label distance=0.1cm]0:{\small GV path}}] at (2.1,5.3) {};
\node [route,color=red,label={[label distance=0.1cm]0:{\small UAV path}}] at (3.9,5.3) {};
\end{tikzpicture}\label{fig:gpcons2}}
\subfigure[This edge includes two maneuvers in a feasible CAGVRP solution; the UAV traversing the arc from $i$ to $j$ and the ground vehicle carrying the UAV and traversing the edge $(j,k)$. The second configuration in this edge has to be a hub and the first configuration should not be a hub in this type of edge. The cost of this edge is $d_{ij}+c_{jk}$.]{\begin{tikzpicture}
\tikzstyle{target}=[circle,draw=green,fill=green!30,inner sep=0.4mm]
\tikzstyle{gv}=[star,star points=4,fill=blue!30,draw=blue,inner sep=0.4mm]
\tikzstyle{av}=[star,star points=6,fill=red!30,draw=red,inner sep=0.4mm]
\tikzstyle{route}=[rectangle, minimum height=0pt, minimum width=6pt, inner sep=0pt, draw]
\draw [thin] (-2,3) rectangle (6,5);
\draw [black,dashed] (2,5) -- (2,3);
\node [draw=none] at (-1,3.3) {\small $\mathcal C(g_i,a_j)$};
\node [draw=none] at (3,3.3) {\small $\mathcal C(g_k,a_k)$};
\draw (-1,4.5) node (u1) [target] {\small $i$};
\draw (0.5,4.5) node (u2) [target] {\small $j$};
\draw (1,3.5) node (u3) [target] {\small $k$};
\draw (3,4.5) node (v1) [target] {\small $i$};
\draw (4.5,4.5) node (v2) [target] {\small $j$};
\draw (5,3.5) node (v3) [target] {\small $k$};
\draw [red,->,thick] (v1) -- (v2);
\draw [blue,->,thick] (v2) -- (v3);
\node [gv] at (0.9,4.5) {};
\node [gv] at (5.4,3.7) {};
\node [av] at (5.4,3.3) {};
\node [av] at (-1.4,4.5) {};
\node [target,inner sep=0.7mm,label={[label distance=0.1cm]0:{\small Target}}] at (-1.8,5.3) {};
\node [gv,label={[label distance=0.1cm]0:{\small GV}}] at (-0.3,5.3) {};
\node [av,label={[label distance=0.1cm]0:{\small UAV}}] at (0.8,5.3) {};
\node [route,color=blue,label={[label distance=0.1cm]0:{\small GV path}}] at (2.1,5.3) {};
\node [route,color=red,label={[label distance=0.1cm]0:{\small UAV path}}] at (3.9,5.3) {};
\end{tikzpicture}\label{fig:gpcons3}}
\caption{Possible edges in the graph $\overline G$. GV abbreviates the ground vehicle.}
\label{fig:gpcons}
\end{figure}
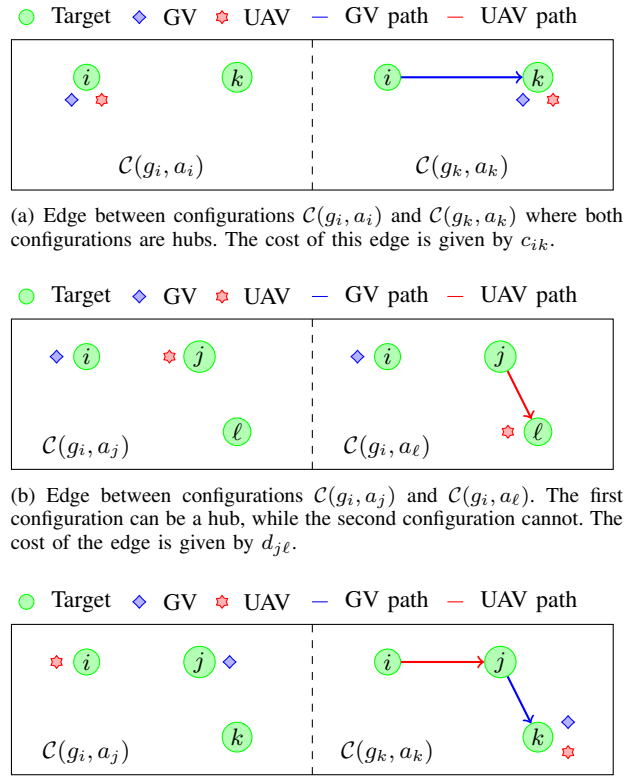

The vertex set, $\overline V$, in the transformed graph, $\overline G$, is defined to be the set of all feasible configurations. We remark that if the communication range was infinite, then $|\overline V| = |T|^2$. But in practice, the communication range could be much less than the maximum distance between any two targets, resulting in $|\overline V| \ll |T|^2$. As for the edge set $\overline E$ of the transformed graph $\overline G$, an edge between pair of vertices corresponding to two feasible configurations $\mathcal C(g_i, a_j)$ and $\mathcal C(g_k, a_l)$ is added if one of the following conditions holds.
\begin{enumerate}
    \item The configurations $\mathcal C(g_i, a_j)$ and $\mathcal C(g_k, a_{\ell})$ are distinct hubs \emph{i.e.}, $i=j$, $k={\ell}$, and $i\neq k$. In this case, the travel cost for the edge joining the two configuration is given by the ground vehicle's travel cost to go from vertex $i$ to $k$ \emph{i.e.}, $c_{ik}$, where $(i,k) \in E$ (see Fig. \ref{fig:gpcons1}).
    \item The ground vehicle is located at the same target in both configurations, the UAV travels from target $j$ to $\ell$, and the configuration $\mathcal C(g_k, a_{\ell})$ is not a hub \emph{i.e.}, $i=k$, $j \neq \ell$, and $k\neq \ell$ (see Fig. \ref{fig:gpcons2}). The cost of the edge in this case is given by $d_{j{\ell}}$, the cost incurred by the UAV to traverse the arc $[j,{\ell}] \in A$. 
    \item The first configuration $\mathcal C(g_i, a_j)$ is not a hub, while the second configuration $\mathcal C(g_k, a_k)$ is a hub with the ground vehicle position changing between the two configurations \emph{i.e.} $i\neq k$ (see Fig. \ref{fig:gpcons3}). An edge between these two configurations in the transformed graph corresponds to two maneuvers in a feasible CAGVRP solution, the first maneuver being the UAV traversing the arc $[i,j] \in A$ and the second one being the ground vehicle carrying the UAV and traversing the edge $(j,k) \in E$. Hence, the cost associated with this edge is given by $d_{ij} + c_{jk}$. 
\end{enumerate}
We remark that the graph $\overline G$ is not a complete graph. Nevertheless, $\overline G$ can be converted to a complete graph by adding the remaining edges with an arbitrary high travel cost. It remains to define the exhaustive partition of the vertex sets for the GTSP on the transformed graph. To that end, for each target $t \in T$, we define the subset of vertices $V_t \subseteq \overline V$ as $V_t\triangleq \{\mathcal C(g_i, a_j) \in \overline V : j = t\}$. It is easy to observe that all the sets $V_t$, $t\in T$, form an exhaustive partition of the set $\overline V$ \textit{i.e.}, $\overline V = \bigcup_{t \in T} V_t$ and $V_i \bigcap V_j = \emptyset$, for every distinct $i,j \in T$. The transformed graph $\overline G=(\overline V, \overline E)$ together with the sets $V_t$, $t \in T$ define the graph on which the GTSP is solved. 

\subsection{Proof of correctness of the transformation algorithm} \label{subsec:tranformation-proof}
Now, in the following two propositions, we present the proof of correctness of the transformation algorithm that transforms the CAGVRP on the graph $G$ to an equivalent GTSP on the transformed graph $\overline G$.

\begin{proposition}\label{prop:gtsp}
Any feasible solution, $\mathcal E_1$, to the GTSP defined on the transformed graph $\overline G = (\overline V, \overline E)$ with a cost of $\mathcal O$ has a corresponding feasible CAGVRP solution on the graph $G = (V,E\bigcup A)$ with the same cost, $\mathcal O$.
\end{proposition}
\begin{proof}
The feasible solution $\mathcal E_1$ to the GTSP starts at the configuration $\mathcal C(g_0, a_0)$, visits a subset of configurations in the vertex set $\overline V$ such that exactly one configuration in each set $V_t$, $t\in T$ is visited, and returns to the configuration $\mathcal C(g_0,a_0)$. The definition of the partition sets $V_t$, $t\in T$, guarantees that the UAV is physically present at every target $i \in T$, either with the ground vehicle (corresponds to a configuration that is a hub) or separately from the ground vehicle. Furthermore, all the vertices in the transformed graph correspond to feasible configurations \emph{i.e.}, the positions of the UAV and the ground vehicle are such that they satisfy the communication constraints in the CAGVRP. Hence, mapping the sequence of edges in the solution $\mathcal E_1$ of the GTSP to the corresponding sequence of maneuvers of the UAV and the ground vehicle on the graph $G$ would yield a feasible CAGVRP solution. The manner in which the edges are constructed and the cost function is defined for each edge in the set $\overline E$ ensure that the cost of the CAGVRP solution, constructed in the aforementioned way, equals $\mathcal O$, completing the proof.
\end{proof}

\begin{proposition}\label{prop:cagvrp}
A feasible solution to the GTSP on the transformed graph $\overline G(\overline V,\overline E)$, $\mathcal E_1$, can be constructed from any feasible solution $\mathcal E_2$ to the CAGVRP on the graph $G(V,E)$. Furthermore, the construction can be performed such that the cost of the solution $\mathcal E_1$ on the graph $\overline G$ is equal to the cost of the solution $\mathcal E_2$ on $G$.
\end{proposition}
\begin{proof}
Let $\mathcal E_2$ be any feasible solution to the CAGVRP. Then, $\mathcal E_2$ can be subdivided into a sequence of maneuvers, each of which could be from the following three maneuvers:
\begin{enumerate}
    \item the ground vehicle carries the UAV and travels from one target to another, 
    \item the ground vehicle stops at a target $i\in T$ and the UAV travels from a target $j\in T$ ($j$ can be equal to $i$) to a target $k \in T$ such that the targets $j$ and $k$ are within the communication distance of $R$ units from the target $i$ \emph{i.e.}, $d(i,j) \leqslant R$ and $d(i,k) \leqslant R$,
    \item the UAV returns to the ground vehicle that stopped at target $i\in T$ from another target $j \in T$, and the ground vehicle carrying the UAV travels to target $k \in T$. 
\end{enumerate}
Notice that, the positions of the vehicles before and after each of the above maneuvers correspond to a feasible configuration in the set $\overline V$, and that each maneuver defined above corresponds to one of the three types of edges in the edge set $\overline E$. Let $\mathcal C_G$ denote all the set of configurations that occur in the solution $\mathcal E_2$. Any feasible solution, $\mathcal E_2$, to the CAGVRP satisfies the Eqs. \eqref{eq:outdegw} and \eqref{eq:indegw}, which ensures for every target $t \in T$, there exists one configuration in $\mathcal C_G$, where the UAV is at the target $t$. Hence, $\mathcal C_G$ satisfies the condition $\mathcal C_G \bigcap V_t = 1$ for every $t \in T$ ($V_t$ is the partition set associated with $t$ in the graph $\overline G$). Therefore, mapping the sequence of maneuvers obtained in $\mathcal E_2$ to the corresponding configurations and edges in the transformed graph $\overline G$ would yield a feasible GTSP solution $\mathcal E_1$. The construction of the edge set $\overline E$, with their cost function, will ensure that the cost of $\mathcal E_1$ in $\overline G$ is equal to the cost of the solution $\mathcal E_2$ in the graph $G$.
\end{proof}

To summarize the propositions \ref{prop:gtsp} and \ref{prop:cagvrp}, for any solution to the GTSP, there is a corresponding feasible solution to the CAGVRP that has the same cost; and for any solution to the CAGVRP, there exists an equivalent feasible solution to the transformed GTSP. Therefore, an optimal solution to the GTSP maps to an optimal solution to the CAGVRP.

\subsection{An illustration of the transformation algorithm} \label{subsec:illustration}
For the sake of clarity, we now present a feasible solution to a CAGVRP instance with its corresponding GTSP solution in the Fig. \ref{fig:cag_gtsp}. In the instance shown in Fig. \ref{fig:cagsoln}, the targets $0$ and $4$ are not within the communication range of any target. The cost of the CAGVRP solution in Fig. \ref{fig:cagsoln} is given by $(c_{01}+c_{14}+c_{40} + d_{12}+d_{23}+d_{31})$. The corresponding feasible solution on the transformed graph $\overline G = (\overline V, \overline E)$ is shown in the Fig. \ref{fig:gtspsoln} along with the costs of the edges. 
\begin{figure}[htpb]
\centering{}
\subfigure[A feasible solution to the CAGVRP]{\begin{tikzpicture}[scale=0.9]
\draw [thin] (-2,-0.5) rectangle (6.5,5.5);
\SetVertexSimple[Shape=circle, FillColor=green!30, MinSize=0.5pt, LineColor=green]
\tikzset{VertexStyle/.style = {shape=circle,draw, inner sep=0.5mm,color=green,fill=green!30}}
\SetVertexLabel
\SetGraphUnit{1.5}
\Vertex[L={\color{black}\small $0$}]{1}
\NOEA[unit = 3.5,L={\color{black}\small $4$}](1){2}
\WE[unit=2.5,L={\color{black}\small $1$}](2){3}
\tikzset{EdgeStyle/.style = {<-,blue,thick}}
\Edges(1,2,3,1)
\NOWE[L={\color{black}\small $2$}](3){4}
\NO[L={\color{black}\small $3$}](3){5}
\tikzset{EdgeStyle/.style={->,red,thick}}
\Edges(3,4,5,3)

\tikzstyle{target}=[circle,draw,color=green,fill=green!30,inner sep=0.7mm]
\tikzstyle{route}=[rectangle, minimum height=0pt, minimum width=6pt, inner sep=0pt, draw]
\matrix [draw=none,below left] at (6,2) {
  \node [target,label={[label distance=0.1cm]0:{\small Target}}] {}; \\
  \node [route,draw=red,label={[label distance=0.1cm]0:{\small UAV path}}] {}; \\
  \node [route,draw=blue,label={[label distance=0.1cm]0:{\small Ground vehicle path}}] {}; \\
};
\end{tikzpicture}\label{fig:cagsoln}}
\subfigure[GTSP solution constructed from the feasible CAGVRP solution in \ref{fig:cagsoln}.]{\begin{tikzpicture}
\tikzstyle{target}=[rectangle,draw=orange,fill=orange!30,inner sep=0.4mm]
\tikzstyle{gv}=[star,star points=4,fill=blue!30,draw=blue,inner sep=0.4mm]
\tikzstyle{av}=[star,star points=6,fill=red!30,draw=red,inner sep=0.4mm]
\tikzstyle{route}=[rectangle, minimum height=0pt, minimum width=6pt, inner sep=0pt, draw]
\draw [thin] (0,0) rectangle (8,7.5);
\node [target] at (2,1) (0_0) {\small $\mathcal C(g_0, a_0)$};
\draw [dash dot,very  thin] (1.25,0.7) rectangle (2.75,1.3);
\node  [draw=none] at (2,0.5) {\small $V_0$};
\node [target] at (1.5,3) (2_1) {\small $\mathcal C(g_2, a_1)$};
\node [target] at (3.1,3) (1_1) {\small $\mathcal C(g_1, a_1)$};
\node [target] at (2,3.6) (3_1) {\small $\mathcal C(g_3, a_1)$};
\draw [dash dot, very thin] (0.75,2.7) rectangle (3.85,3.9);
\node  [draw=none] at (1.5,4.1) {\small $V_1$};
\node [target] at (1.5,6) (2_2) {\small $\mathcal C(g_2, a_2)$};
\node [target] at (3.1,6) (1_2) {\small $\mathcal C(g_1, a_2)$};
\node [target] at (2,6.6) (3_2) {\small $\mathcal C(g_3, a_2)$};
\draw [dash dot, very thin] (0.75,5.7) rectangle (3.85,6.9);
\node  [draw=none] at (2,7.1) {\small $V_2$};
\node [target] at (5,4.5) (1_3) {\small $\mathcal C(g_1, a_3)$};
\node [target] at (6.6,4.5) (2_3) {\small $\mathcal C(g_2, a_3)$};
\node [target] at (6,5.1) (3_3) {\small $\mathcal C(g_3, a_3)$};
\draw [dash dot, very thin] (4.25,4.2) rectangle (7.35,5.4);
\node  [draw=none] at (6,5.6) {\small $V_3$};
\node [target] at (6,2) (4_4) {\small $\mathcal C(g_4, a_4)$};
\draw [dash dot,very  thin] (5.25,1.7) rectangle (6.75,2.3);
\node  [draw=none] at (6,1.5) {\small $V_4$};
\draw [->,thick] (0_0) -- node [midway,draw=none,label={[label distance=0.05cm]0:{\small $c_{01}$}}] {} (1_1);
\draw [->,thick] (1_1) -- node [midway,draw=none,label={[label distance=-0.8cm]0:{\small $d_{12}$}}] {} (1_2);
\draw [->,thick] (1_2) -- node [midway,draw=none,label={[label distance=-0.8cm]0:{\small $d_{23}$}}] {} (1_3);
\draw [->,thick] (1_3) -- node [midway,draw=none,label={[label distance=0.05cm]0:{\small $(d_{13}+c_{14})$}}] {} (4_4);
\draw [->,thick] (4_4) -- node [below] {\small $c_{40}$} (0_0);
\end{tikzpicture}\label{fig:gtspsoln}}
\caption{A solution to the GTSP on $\overline G$ constructed from the solution of the CAGVRP. The costs of both the solutions in their respective graphs is given by $(c_{01}+c_{14}+c_{40} + d_{12}+d_{23}+d_{31})$.}
\label{fig:cag_gtsp}
\end{figure}
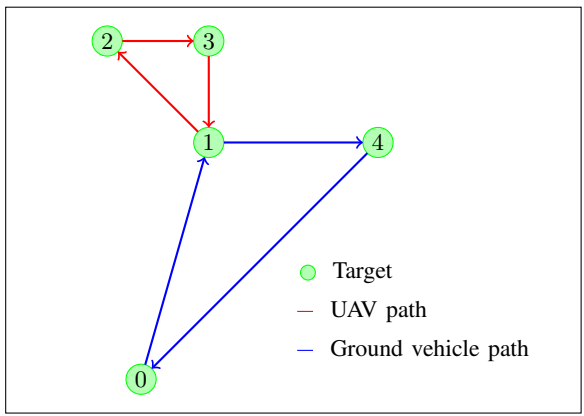
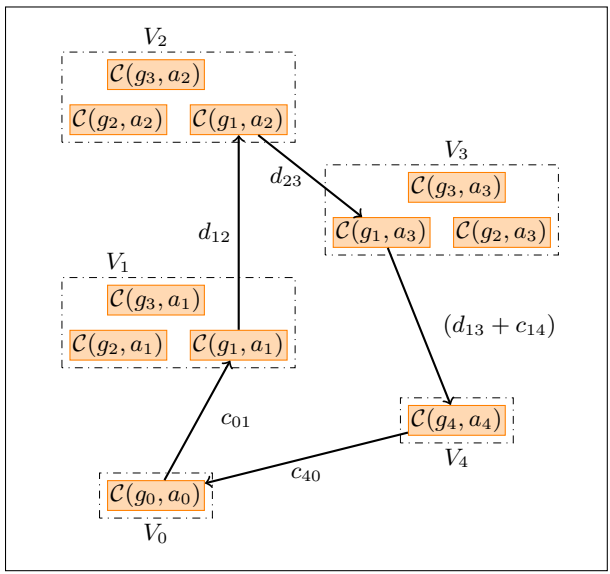
For each target $i$, the set $V_i$ in Fig. \ref{fig:gtspsoln} define the partition of the vertex set $\overline V$. The two vehicles are located at target $0$ initially, and they together travel to target $1$. This corresponds to a maneuver from $\mathcal{C}(g_0,u_0)$ to $\mathcal{C}(g_1,u_1)$ in Fig. \ref{fig:gtspsoln}. At target $1$, the UAV is deployed and it visits targets $2$ and $3$ (in that order), while ground vehicle is at target $1$. The corresponding configurations visited in the GTSP solution are $\mathcal{C}(g_1,u_2)$ and $\mathcal{C}(g_1,u_3)$. The UAV then returns to the ground vehicle at target $1$, and they together visit target $4$ and return to target $0$. In the next section, we present extensive computational results to corroborate the effectiveness of all the algorithms proposed for the CAGVRP. 

\section{Computational Results} \label{sec:results}
The branch-and-cut algorithm that is used to solve the CAGVRP to optimality was implemented using CPLEX $12.6$ \cite{cplexmanual} with C++ API. The internal CPLEX cut-generation routines were disabled and CPLEX was used only to manage the enumeration tree. The transformation algorithm was also implemented using C++, and the heuristic presented in \cite{Smith2016} was used to solve the GTSP on the transformed graph. All the simulations were performed on a Macbook Pro with an Intel Core i5, 2.7 GHz processor. The performance of the algorithms were tested on three different classes of randomly generated instances. The procedure for generating the instances is given below. A time limit of 6 hours was imposed on every run of the branch-and-cut algorithm to obtain an optimal solution. \\

\noindent \textit{\underline{\smash{Instance generation:}}} Three classes of randomly generated CAGVRP instances (class A, class B, and class C) were used to test the algorithms presented in this article. For all the classes of the instances the locations of the targets were randomly generated in a $100 \times 100$ grid. A communication range of $25$ units was used for all the instances. Instances in classes A and B consisted of small- and medium-sized instances \emph{i.e.}, $|T| \in \{20, 30, 40, 50\}$ and the instances in class C consisted of large-sized test instances \emph{i.e.}, $|T| \in \{60, 70, 80\}$. For all the instances, the travel cost for the ground vehicle to travel between any pair of targets $i,j\in T$ is chosen to be equal to the Euclidean distance, $l_{ij}$, between the two targets; and the travel cost of the UAV is set to $\alpha\cdot l_{ij}$, where $\alpha$ is a scaling factor in the set $\{0.1, 0.2, 0.3\}$. The target locations for the class A and class C instances were randomly generated from uniform distribution in the $100 \times 100$ grid. As for the instances in class B, the target locations were randomly chosen to form clusters in the grid with a lower bound on the separation distance between any pair of targets belonging to two different clusters. For every class and every $|T|$ in the corresponding class, $20$ instances were generated. In total, for each value of $\alpha \in \{0.1, 0.2, 0.3\}$, there were $80$, $80$, and $60$ instances corresponding to the classes A, B, and C respectively.  

\subsection{Performance of the branch-and-cut algorithm} \label{subsec:bandc_times}
We first examine the scalability aspect of the branch-and-cut algorithm in computing an optimal solution for the instances in classes A and B. To that end, the Table \ref{tab:succ} presents the number of instances out of $20$, the branch-and-cut algorithm is able to compute the optimal solution within the computation time limit of 6 hours. It is clear from the table that the branch-and-cut algorithm is capable of solving instances with up to $40$ targets effectively, and the $50$-target instances with a little difficulty. 

\begin{table}[htbp]
\centering
\caption{Number of instances out of $20$ that are solved to optimality.}
\begin{tabular}{ccccc}
\toprule
Instance class & $|T|$ & $\alpha = 0.1$ & $\alpha = 0.2$ & $\alpha = 0.3$ \\
\midrule 
\multirow{4}{*}{A} & 20 & 20 & 20 & 20 \\
& 30 & 20 & 20 & 20 \\
& 40 & 20 & 19 & 19 \\
& 50 & 14 & 14 & 7 \\
\midrule 
\multirow{4}{*}{B} & 20 & 20 & 20 & 20 \\
& 30 & 20 & 20 & 20 \\
& 40 & 17 & 18 & 13 \\
& 50 & 17 & 17 & 17 \\
\bottomrule
\end{tabular}
\label{tab:succ}
\end{table}

The run time for the instances is provided by the box plots shown in Fig. \ref{fig:box}. These box plots show the computation time for all the class A and class B instances that were solved to optimality within the stipulated time limit. The computation times indicate that instances with up to $40$ targets could be solved to optimality by the branch-and-cut algorithm in around $1000$ seconds and it takes around $5000$-$10000$ seconds to solve the $50$-target instances.  In general, we observe that the instances with a scale factor of $\alpha = 0.3$ are more difficult to solve, and need more computation time. This is expected for the following reason: suppose the cost of travel by the UAV is zero, then the algorithm would try to assign as many targets as possible to the UAV. As the cost of travel by the UAV is increased, it has to find the right partitioning of the targets to be assigned to the ground vehicle and the UAV, that minimizes the total cost. This additional partitioning decisions could potentially increase the combinatorial complexity of finding the optimal solution.

\begin{figure}[htbp]
    \centering
    \subfigure[Computation time for the instances in class A that were solved to optimality within the 6-hour limit.]{\includegraphics[scale=0.6]{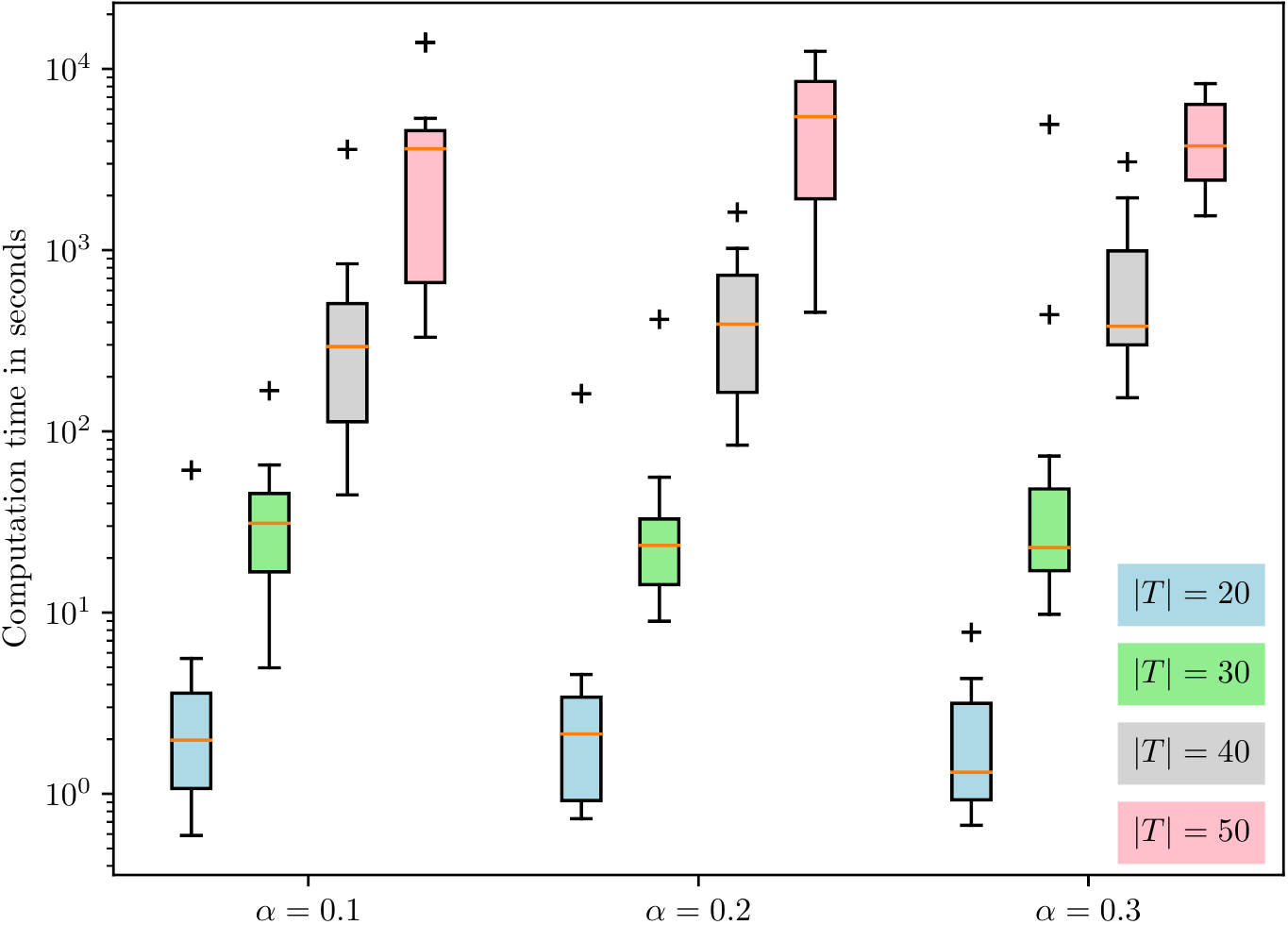}}
    \label{subfig:box_a}
    \subfigure[Computation time for the instances in class B that were solved to optimality within the 6-hour limit.]{\includegraphics[scale=0.6]{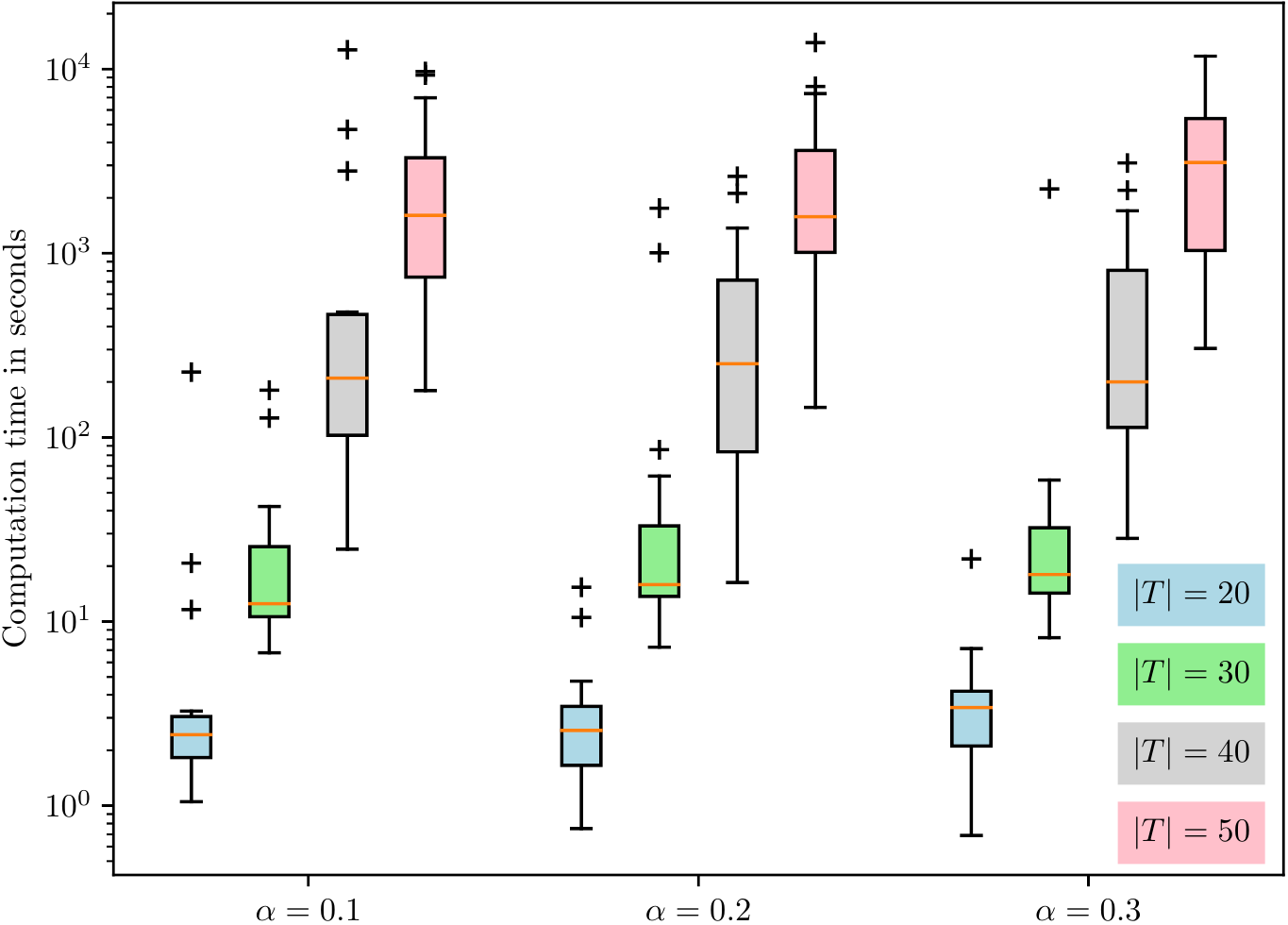}}
    \label{subfig:box_b}
    \caption{Time taken to compute an optimal solution by the branch-and-cut algorithm.}
\label{fig:box}
\end{figure}

The next set of results in Table \ref{tab:cuts} illustrate the effect of the dynamically adding the connectivity and $2$-matching constraints in Eqs. \eqref{eq:xsec} -- \eqref{eq:wsec2} and \eqref{eq:twomatch}, respectively, to the branch-and-cut algorithm by utilizing the separation algorithms detailed in Sec. \ref{subsec:separation}. These results are only shown for the class A instances since the numbers are similar for the class B instances. 

\begin{table}[htbp]
    \centering
    \caption{Average and maximum number of connectivity constraints in Eqs. \eqref{eq:xsec} -- \eqref{eq:wsec2} and $2$-matching constraints in Eq. \eqref{eq:twomatch} that were dynamically added to the branch-and-cut algorithm for the class A instances.}
    \begin{tabular}{crr}
        \toprule
         $|T|$ & Average & Maximum  \\
        \midrule
         20 & 399.21 & 4590 \\
         30 & 2388.00 & 31851\\
         40 & 5529.80 &  42789\\
         50 & 12931.48 & 36482 \\
         \bottomrule
    \end{tabular}
    \label{tab:cuts}
\end{table}
Finally, the Fig. \ref{fig:lp_round} shows the number of times the LP rounding heuristic in Sec. \ref{subsec:lprounding} produced better feasible solutions than the ones generated by CPLEX during the run of the branch-and-cut algorithm. 
\begin{figure}[htbp]
    \centering
    \includegraphics[scale=0.6]{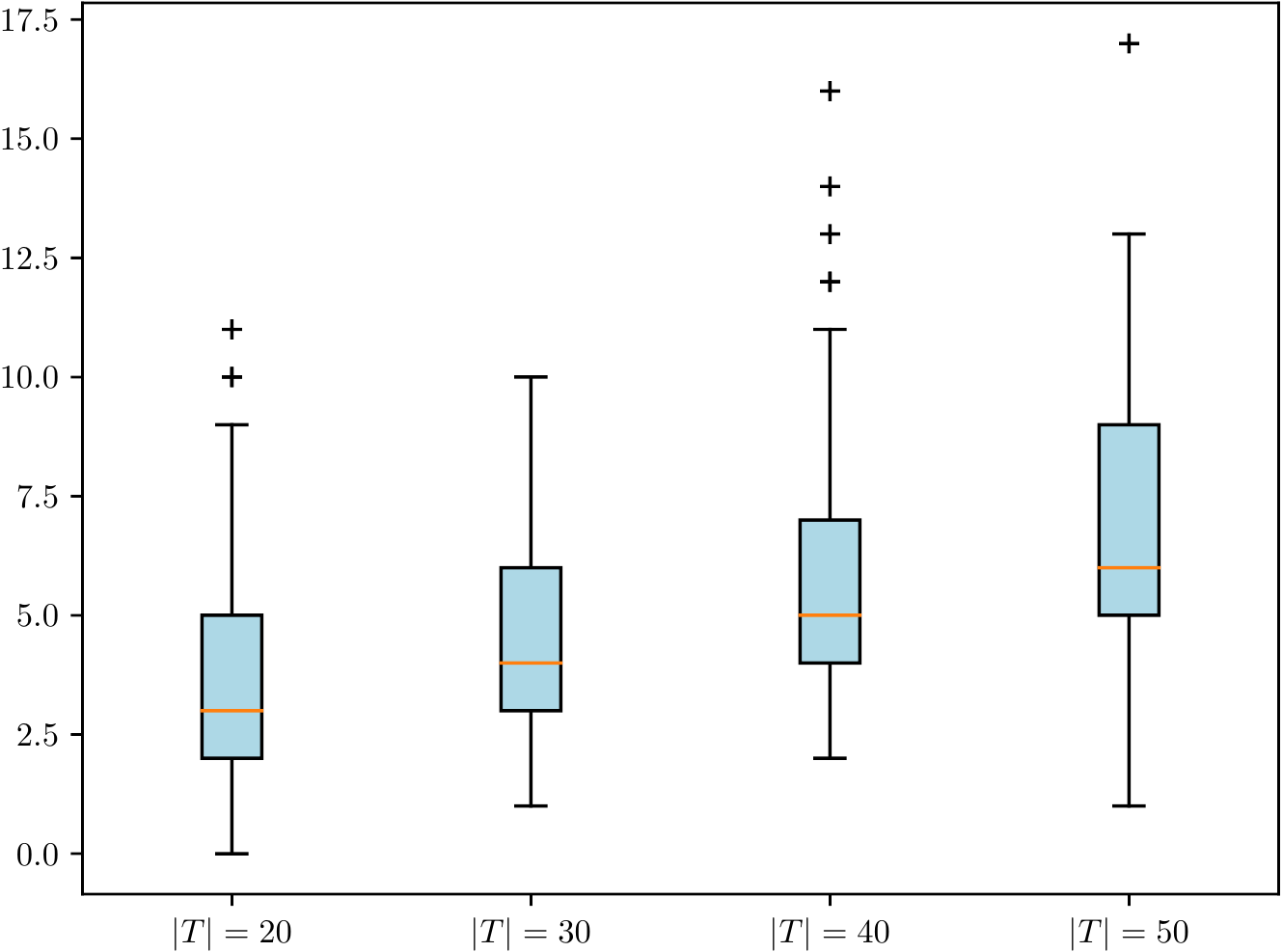}
    \caption{Box plot of the number of times the LP rounding heuristic in Sec. \ref{subsec:lprounding} was able to produce better feasible solutions than the feasible solutions computed by CPLEX.}
    \label{fig:lp_round}
\end{figure}
Next, we present the computational results for the transformation algorithm and corroborate its effectiveness on all the classes of instances.

\subsection{Performance of the transformation algorithm} \label{subsec:gtsp}
\begin{table}[htbp]
    \centering
     \caption{Relative optimality gaps (in \%) produced by the transformation algorithm on the instances in class A that were solved to optimality.}
    \begin{tabular}{cccc}
        \toprule
        $|T|$ & Maximum & Average & Standard deviation \\ 
        \midrule
        20 & 0.60 & 0.18 & 0.17 \\
        30 & 3.05 & 0.50 & 0.66 \\
        40 & 3.01 & 0.67 & 0.63 \\
        50 & 2.72 & 0.62 & 0.49 \\ 
        \bottomrule
    \end{tabular}
    \label{tab:opt-gap-A}
\end{table}
\begin{table}[htbp]
    \centering
     \caption{Relative optimality gaps (in \%) produced by the transformation algorithm on the instances in class B that were solved to optimality.}
    \begin{tabular}{cccc}
        \toprule
        $|T|$ & Maximum & Average & Standard deviation \\ 
        \midrule
        20 & 1.09 & 0.39 & 0.26 \\
        30 & 1.17 & 0.49 & 0.23 \\
        40 & 1.60 & 0.67 & 0.35 \\
        50 & 1.78 & 0.70 & 0.45 \\ 
        \bottomrule
    \end{tabular}
    \label{tab:opt-gap-B}
\end{table}
As mentioned in Sec. \ref{sec:trans_alg}, once the instance of CAGVRP is transformed to an instance of GTSP, the transformed problem is solved using the large neighborhood search heuristic for GTSP \cite{Smith2016}. Hence, all the solutions produced by the transformation algorithms are heuristic solutions and have a cost greater than or equal to the optimal. To show the effectiveness of the algorithm in computing solutions whose objective value is close to the cost of the optimal solution, we first present the results of the transformation algorithm on the instances of classes A and B, for which the optimum is computed using the branch-and-cut algorithm. To that end, the Fig. \ref{fig:heuristic_quality} shows a scatter plot of the optimal solution objective value vs the objective values obtained using the transformation algorithm. %
\begin{figure}[htbp]
    \centering
    \subfigure[Optimal cost vs heuristic solution cost for instances in class A.]{\includegraphics[scale=0.6]{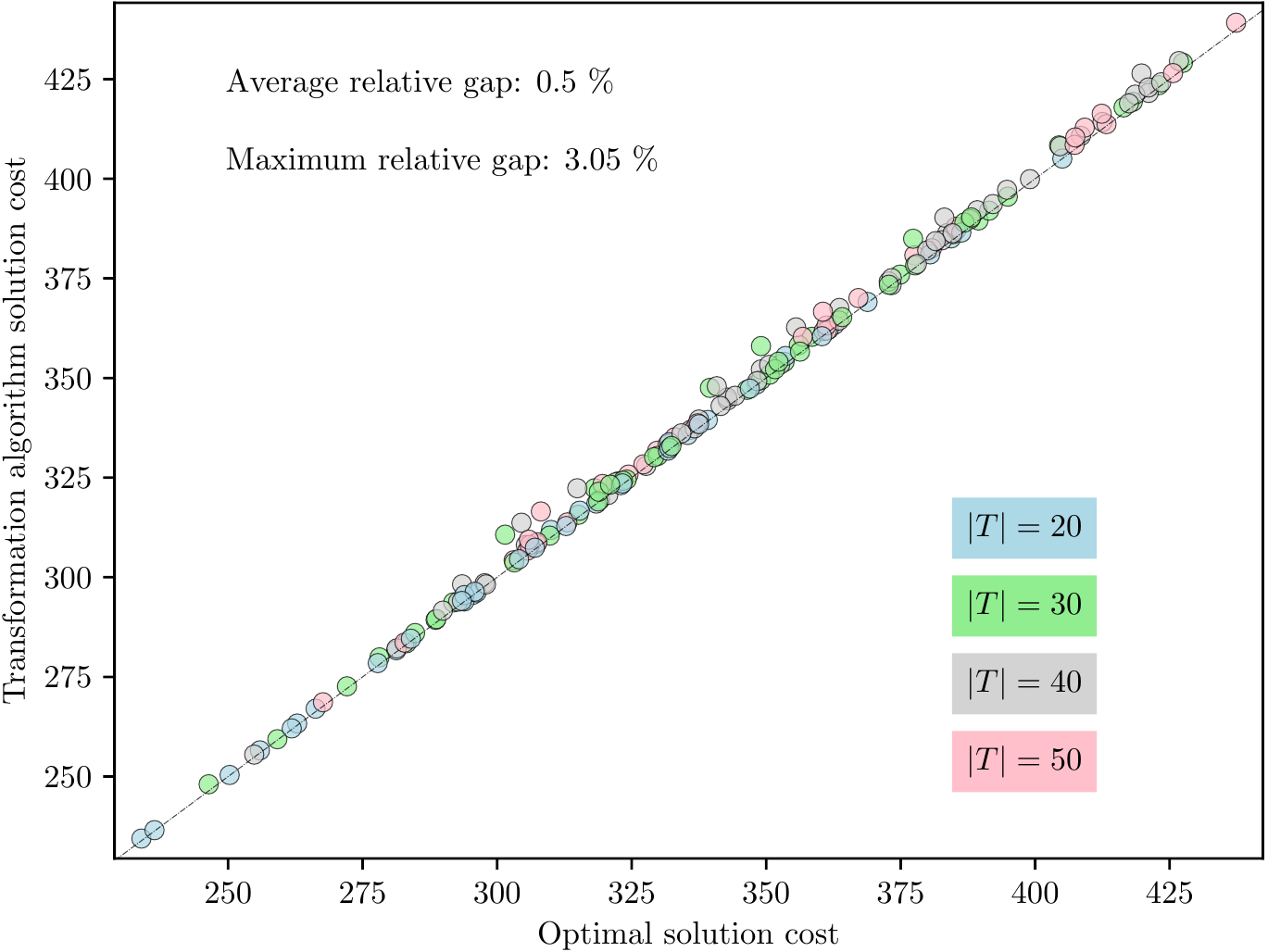}}
    \label{subfig:h_a}
    \subfigure[Optimal cost vs heuristic solution cost for instances in class B.]{\includegraphics[scale=0.6]{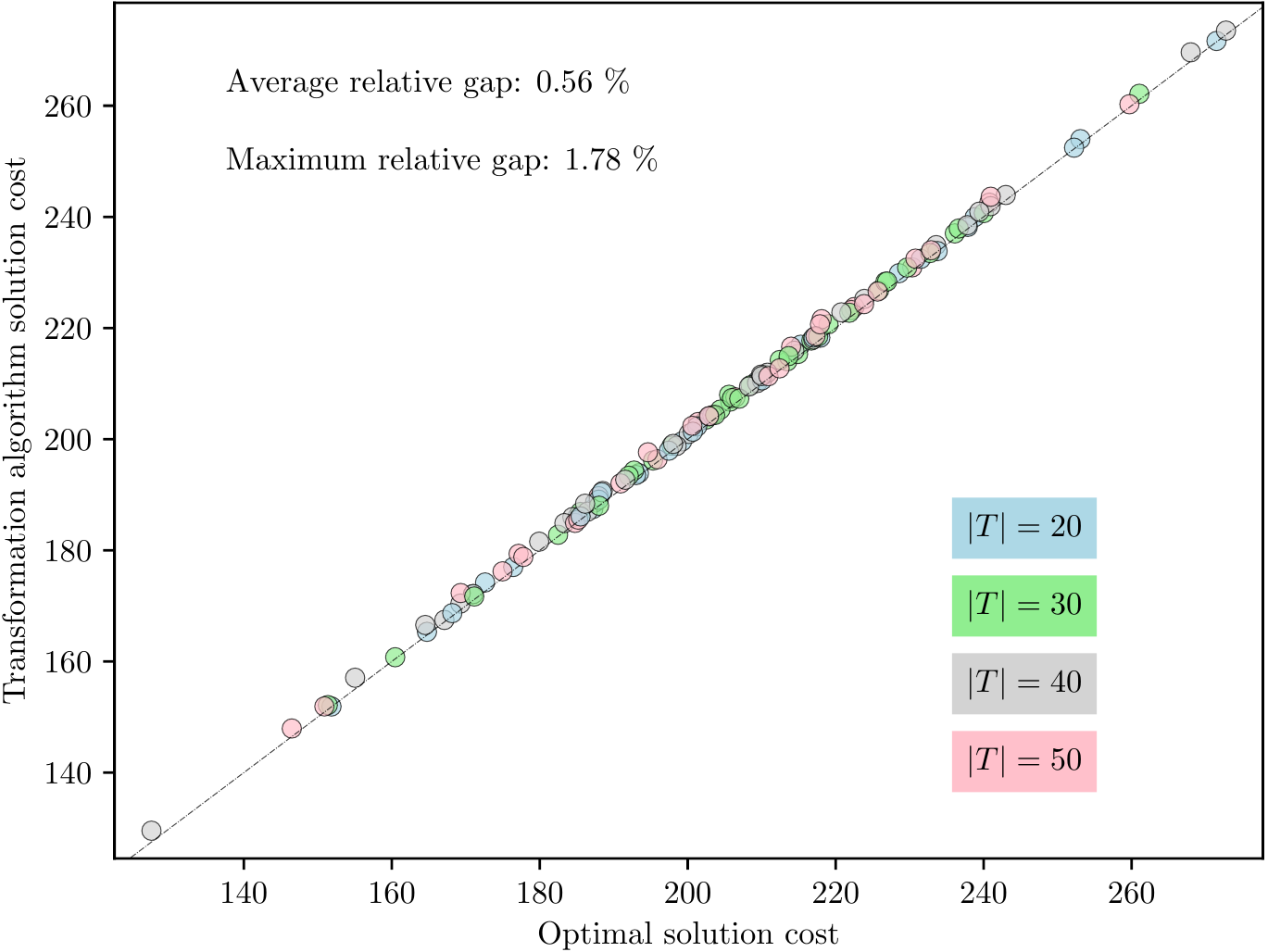}}
    \label{subfig:h_b}
    \caption{The scatter plots only include instances that were solved to optimality.}
\label{fig:heuristic_quality}
\end{figure}
The scatter plots in Fig. \ref{fig:heuristic_quality} shows that the transformation algorithm is very effective in computing heuristic solutions that are very close to the optimal with an average relative gap (see Eq. \ref{eq:relativegap}) of $0.5\%$ and $0.56\%$ for the the instances in classes A and B, respectively. In Fig. \ref{fig:heuristic_quality}, the closer the points are towards to line of slope $1$ unit, the better the quality of the solutions produced by the transformation algorithm. Furthermore, the Tables \ref{tab:opt-gap-A} and \ref{tab:opt-gap-B} show the maximum, average, and standard deviation of the relative gap (in percentage) between the objective values obtained by the exact algorithm, denoted by $c^*$, and transformation algorithm, denoted by $c^t$, for the instances in class A and B, respectively. The mathematical definition of relative gap is given by the following equation:
\begin{equation}
    \text{Relative Gap} = \frac{c^t - c^*}{c^*} \cdot 100 \%. \label{eq:relativegap}
\end{equation}

As for the instances in class C, we present a histogram of the computation times taken by the transformation algorithm in Fig. \ref{fig:histogram}. We observe from the histogram that the transformation algorithm is successful in finding a heuristic solution to the majority of the instances in class C within $50$ seconds. Finally, the table \ref{tab:time-C} shows the maximum, average, and standard deviation of the computation time of the transformation algorithm to compute a feasible solution for every instance in classes A, B, and C. 
\begin{figure}
    \centering
    \includegraphics[scale=0.6]{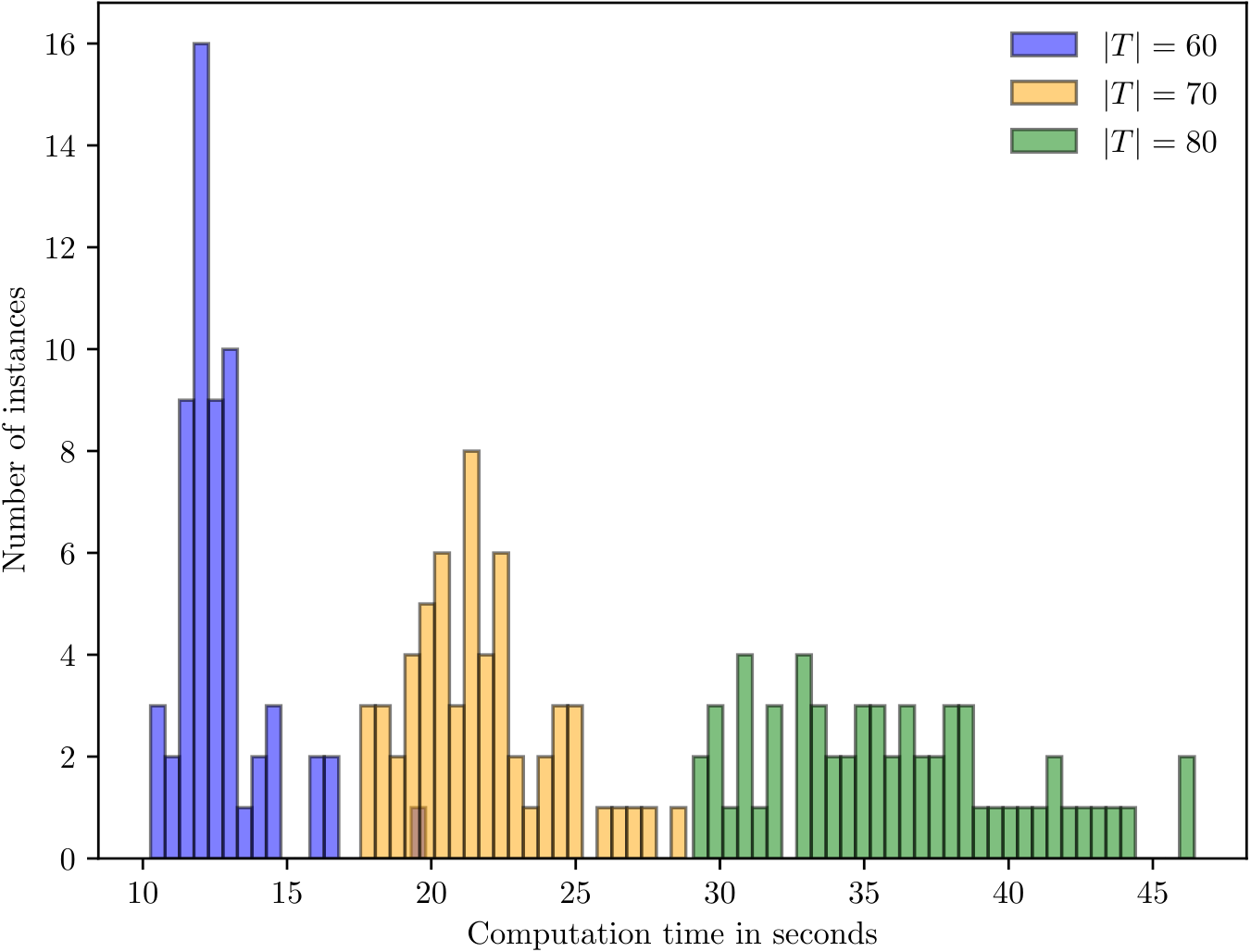}
    \caption{Histogram of the computation time taken by the transformation algorithm to find a heuristic solution to the instances in class C.}
    \label{fig:histogram}
\end{figure}
\begin{table}[htbp]
    \centering
     \caption{Computation time taken (in seconds) by the transformation algorithm for all the instances.}
    \begin{tabular}{cccc}
        \toprule
        $|T|$ & Maximum & Average & Standard deviation \\ 
        \midrule
        20 	 & 1.42 & 	 1.24 	& 0.08 \\
        30 	 & 2.38 & 	 1.91 	& 0.17 \\
        40 	 & 4.33 &	 3.46 	& 0.45 \\
        50 	 & 8.46 & 	 6.86 	& 0.81 \\
        60 	 & 19.80 & 	 12.72 	& 1.65 \\
        70 	 & 28.82 & 	 21.71 	& 2.52 \\
        80 	 & 46.44 &	 36.03 	& 4.38 \\
        \bottomrule
    \end{tabular}
    \label{tab:time-C}
\end{table}

\section{Conclusion and future work} \label{sec:conc}
This article develops a mathematical programming formulation, an algorithm to compute the optimal solution, and a fast heuristic for the cooperative vehicle routing problem involving one UAV and one ground vehicle in the presence of communication restrictions. The algorithms that were presented in this paper were tested on a wide range of instances, and the computational results corroborate the effectiveness of branch-and-cut algorithm to compute optimal solutions for small- and medium-sized test instances. The transformation algorithm is effective in computing near-optimal solutions for small-, medium-, and large-sized instances with a moderate computation time. Future directions for this work include generalization of this cooperative routing problem for multiple UAVs and/or multiple ground vehicles.

\section*{Acknowledgements} 
We gratefully acknowledge the support of the U.S. Department of Energy through the LANL/LDRD Program and the Center for Nonlinear Studies for this work.

\bibliographystyle{IEEEtran}
\bibliography{rcpacc.bib}
\end{document}